\documentclass[a4paper,10pt]{article}
\usepackage[dvipsnames]{xcolor}

\usepackage{amssymb}
\usepackage{amsmath}
\usepackage{bbm}
\usepackage{mathrsfs}

\usepackage{hyperref}
\usepackage{bm}

\hypersetup{
    colorlinks,
    linkcolor={red!50!black},
    citecolor={blue!50!black},
    urlcolor={blue!80!black}
}

\usepackage{amsthm}
\usepackage[center]{caption}
\usepackage{enumerate}
\usepackage{verbatim}
\usepackage{authblk}
\usepackage[hmargin=3cm, vmargin=3.5cm]{geometry}

\usepackage{color}
\usepackage{tikz}
\usepackage[normalem]{ulem}
\usepackage[dvipsnames]{xcolor}
\usepackage{verbatim}
\usepackage{fancyhdr}

\newtheorem{thm}{Theorem}[section]
\newtheorem{cor}[thm]{Corollary}
\newtheorem{lem}[thm]{Lemma}

\newcommand{\Q}{\mathbb{Q}}
\renewcommand{\P}{\mathbb{P}}
\newcommand{\E}{\mathbb{E}}
\newcommand{\R}{\mathbb{R}}
\newcommand{\F}{\mathcal{F}}

\newcommand{\ind}{\mathbbm{1}}

\newcommand{\bp}{\begin{proof}}
\newcommand{\ep}{\end{proof}}
\def\bal#1\eal{\begin{align*}#1\end{align*}}
\newcommand{\Nc}{\mathcal{N}}

\renewcommand{\d}{{\rm{d}}}

\newcommand{\tpi}{\tilde{\pi}}
\newcommand{\cN}{\mathcal{N}}

\newcommand{\bga}{ \boldsymbol\gamma }

\numberwithin{equation}{section}

\author{Samuel G.~G.~Johnston}

\affil{\small{School of Mathematics and Statistics, University College Dublin, Belfield, Dublin 4}}

\title{The genealogy of Galton-Watson trees}

\begin{document}
\maketitle

\begin{abstract}
Take a continuous-time Galton-Watson tree and pick $k$ distinct particles uniformly from those alive at a time $T$. What does their genealogical tree look like? The case $k=2$ has been studied by several authors, and the near-critical asymptotics for general $k$ appear in Harris, Johnston and Roberts (2018) \cite{HJR17}. Here we give the full picture.
\end{abstract}

\section{Introduction}
Let $L$ be a random variable taking values in $\{0,1,2,\ldots\}$. Consider a continuous-time Galton-Watson tree starting with one initial particle, branching at rate $1$, and with offspring distributed like $L$.
Let $N_t$ be the number of particles alive at time $t$, and write $f(s) := \E[ s^L]$ and $F_t(s) := \E[ s^{N_t} ]$ for the generating functions associated with the process.\\

Let $T > 0$, and on the event $\{N_T \geq k\}$ pick $k$ distinct particles $U_1,\ldots,U_k$ uniformly from those alive at time $T$. For each earlier time $t \in [0,T]$, define the equivalence relation $\sim_t$ on $\{1,\ldots,k\}$ by
\begin{align*}
i \sim_{t} j ~\iff ~\text{$U_i$ and $U_j$ share a common ancestor alive at time $t$}.
\end{align*}
We let $\pi^{k,L,T}_t$ denote the random partition of $\{1,\ldots,k\}$ corresponding to this equivalence relation.
The process $(\pi^{k,L,T}_t)_{t \in [0,T]}$, defined on the event $\{ N_T \geq k \}$, is a right-continuous partition-valued stochastic process characterising the entire genealogical tree of $U_1,\ldots,U_k$. \\

Our goal is to describe the law of $(\pi^{k,L,T}_t)_{t \in [0,T]}$ conditioned on the event $\{N_T \geq k \}$, with a view towards the asymptotic regime $T \to \infty$. We find that as $T \to \infty$, there are marked differences in the qualitative behaviour of $(\pi^{k,L,T}_t)_{t \in [0,T]}$ depending on the mean number of offspring
\begin{align*}
m := f'(1).
\end{align*}
Before we state our results in full generality in Section 3, we give an impression of the structure we expect to encounter by exploring the special case $k=2$, which features as the focus of a chapter in the recent book \cite{Ath16}, and on which the majority of the related literature concentrates.

\section{The case $k=2$}
The case $k=2$ amounts to choosing two particles uniformly from those alive at a time $T$ from a tree with offspring distributed like $L$, and studying the time $\tau^{L,T}$ in $[0,T]$ at which they last shared a common ancestor. In terms of the partition process $(\pi^{2,L,T}_t)_{t \in [0,T]}$, $\tau^{L,T}$ is the time at which the single block $\{\{1,2\}\}$ splits into the pair of singletons $\{ \{1\},\{2\} \}$. \\

The following characterisation of the law of $\tau^{L,T}$ (which we will generalise later) was first given by Lambert \cite{Lam03}.

\begin{lem}[Lambert \cite{Lam03}, Corollary 1] \label{Lambert Lemma} On $\{N_T \geq 2\}$, pick two distinct particles uniformly from those alive at time $T$. Let $\tau^{L,T} \in [0,T]$ be the time at which they last shared a common ancestor. Then
\begin{align} \label{lam_eq}
\P\Big( \tau^{L,T} \in [t,T], N_T \geq 2  \Big) =  \int_0^1 (1-s) \frac{ F''_{T-t}(s) }{ F'_{T-t}(s) } F'_T(s) ds,
\end{align}
where $F_t(s) = \E[s^{N_t}]$.
\end{lem}
Although Lambert's result gives a powerful implicit characterisation of the distribution of $\tau^{L,T}$,  
it is difficult to infer qualitative properties of this random variable directly from \eqref{lam_eq}.
When $T \to \infty$ however, it is possible to gain a more intuitive insight.
Unsurprisingly, different qualitative behaviours arise depending on whether the underlying Galton-Watson tree is supercritical, critical, or subcritical. 
These cases correspond to $m > 1$, $m=1$, and $m < 1$ respectively (where $m = f'(1)$).\\

In the remainder of this section we will exploit classical limit theory of Galton-Watson trees in conjunction with \eqref{lam_eq} to show that conditioned on $\{N_T \geq 2\}$, we have the following limiting behaviour in $\tau^{L,T}$ as $T \to \infty$:
\begin{itemize}
\item When the tree is supercritical, $\tau^{L,T}$ remains near the beginning of the interval $[0,T]$. That is, we have the convergence in distribution
\begin{align*}
\tau^{L,T} \xrightarrow{D} \bar{\tau}^L
\end{align*}
as $T \to \infty$, where $\bar{\tau}^L$ is a $[0,\infty)$-valued random variable depending on the law of $L$.
\item When the tree is critical, $\tau^{L,T}$ grows linearly in $T$. That is, we have the convergence in distribution
\begin{align*}
\tau^{L,T}/T \xrightarrow{D} \bar{\tau}^{\mathsf{crit}}
\end{align*}
as $T \to \infty$, where $\bar{\tau}^{\mathsf{crit}}$ is a $[0,1]$-valued random variable universal in all critical offspring distributions with finite variance.
\item When the tree is subcritical, $\tau^{L,T}$ remains near the end of the interval $[0,T]$. That is, we have the convergence in distribution
\begin{align*}
T - \tau^{L,T} \xrightarrow{D} \bar{\nu}^L
\end{align*}
as $T \to \infty$, where $\bar{\nu}^L$ is a $[0,\infty)$-valued random variable depending on the law of $L$.
\end{itemize}
In all three cases we are able to obtain integral formulas for the law of the limit variables.

\subsection{The case $k=2$, supercritical}
First we consider the supercritical case $m > 1$. B{\"u}hler \cite{buhler:super} first observed that when two particles are chosen uniformly from a supercritical tree at a large time, their most recent common ancestor was a member of one of the first generations in tree. More recently, Athreya \cite{Ath12a} showed that conditioned on $\{N_T \geq 2\}$, $\tau^{L,T}$ converges in distribution to a $[0,\infty)$-valued random variable $\bar{\tau}^L$ as $T \to \infty$. \\

Without too much concern at this stage for technical details, we now outline how it is possible to use limit theory for supercritical trees in conjunction with Lambert's formula \eqref{lam_eq} to obtain a formula for the law of the limit variable $\bar{\tau}^L$.\\

When the tree is supercritical and the \emph{Kesten-Stigum condition} $\E[L \log_+L ] <\infty$ holds, the non-negative and unit-mean martingale $W_t := N_t e^{ - (m-1)t} $ converges to a well behaved limit $W_\infty$ \cite{kesten stigum}. This suggests that at a large time $T$, the population size is of order $e^{(m-1)T}$, and it would be useful to study the scaling $s = e^{ - ve^{ - (m-1)t}}$ in the generating function $F_t(s)$. Indeed, if we let $\varphi(v) := \E [e^{ - vW_\infty} ]$ denote the Laplace transform of the martingale limit $W_\infty$, in Lemma \ref{tech lemma} we will show that 
\begin{align} \label{sibelius}
\lim_{T \to \infty} e^{ - k(m-1)T} F^{(k)}_{T-t}(e^{ - v e^{ - (m-1)T}})  = (-1)^k e^{ - k(m-1)t} \varphi^{(k)}(v e^{- (m-1)t} ), ~~~ k \geq 0,
\end{align}
where $F^{(k)}_t(s) := \frac{ \partial^k}{ \partial s^k} F_t(s)$. Assuming for now we can take the limit inside the integral, with the change of variable $s = e^{ - ve^{ - (m-1)T}}$ in \eqref{lam_eq}, using \eqref{sibelius} in the final equality below we obtain
\begin{align}
\P(  \bar{\tau}^L > t, \text{survival}  ) & := \lim_{T \to \infty} \P \Big( \tau^{L,T} \in [t,T], N_T \geq 2  \Big) \nonumber \\
&= \lim_{T \to \infty}  \int_0^1 (1-s) \frac{ F''_{T-t}(s) }{ F'_{T-t}(s) } F'_T(s) ds \nonumber \\
&= \int_0^\infty v e^{ - (m-1)t} \frac{ \varphi''(v e^{ - (m-1)t})}{ \varphi'(v e^{ - (m-1)t} ) } \varphi'(v) dv  \label{tau lim}.
\end{align} 
The formula \eqref{tau lim} appears to be new, and corresponds to the special case $k=2$ of our main result for supercritical trees, Theorem \ref{MTSuper}.

\subsection{The case $k=2$, critical}

We now move onto the critical case $m=1$, which has received a lot of attention from different authors \cite{Ath12b,Dur78,HJR17,OCo95, Zubkov}. Under the second moment assumption $f''(1) < \infty$, Zubkov \cite{Zubkov} found that conditioned on $\{N_T \geq 2\}$, $\tau^{L,T}/T $ converges in distribution to a $[0,1]$-valued random variable $\bar{\tau}^{\mathsf{Crit}}$ as $T \to \infty$.\\

Like in the supercritical case, it is possible to use limit theory for critical trees in conjunction with \eqref{lam_eq} to obtain the law of $\bar{\tau}^{\mathsf{Crit}}$. Namely, the Kolmogorov-Yaglom exponential limit law \cite[III.7]{AN72} states that for critical trees with finite variance
\begin{align} \label{kolm}
\lim_{T \to \infty} T \P( N_T > 0 ) = \frac{1}{c}, ~~~ \lim_{T \to \infty} \P \left( \frac{N_T}{cT} >  x \Bigg| N_T > 0 \right) = e^{ -x},
\end{align}
where $c := f''(1)/2$. In Lemma \ref{critical scaling}, we use the exponential limit law \eqref{kolm} to show that for $k \geq 1$ and $a \in (0,1]$
\begin{align} \label{sugar}
\lim_{T \to \infty} T^{ - k + 1 } F_{aT}^{(k)}\left( e^{ - \frac{ \theta}{c T } } \right) = \left( ac \right)^{k-1} \frac{ k! }{ (1 + a \theta)^{k+1}} .
\end{align}
For $u \in [0,1]$, set $t = uT$ and take the change of variable $s = \exp \left( - \frac{ \theta}{c T} \right)$ in \eqref{lam_eq}. Assuming we can take the limit inside the integral, using \eqref{sugar} in the third equality below we obtain
\begin{align}
\P(  \bar{\tau}^{\mathsf{Crit}} \in [u,1] ) &:= \lim_{T \to \infty} \P \left( \tau^{L,T}/T \in [u,1]~ |~ N_T \geq 2 \right) \nonumber
\\ &= \lim_{T \to \infty} \frac{1}{\P(N_T \geq 2)}  \int_0^1 (1-s) \frac{ F''_{(1-u)T}(s) }{ F'_{(1-u)T}(s) } F'_T(s) ds  \nonumber
\\ &= \int_0^\infty \frac{ 2(1-u)}{ (1 + (1-u)\theta) } \frac{ \theta}{ ( 1 + \theta)^2} d\theta  \nonumber
\\ &= \frac{ 2(1-u)}{ u^2} \left( \log \left( \frac{1}{1-u} \right) - u \right). \label{criteq}
\end{align}
Various formulas for the law of $\bar{\tau}^{\mathsf{Crit}}$ have appeared in the literature. Durrett \cite{Dur78} gave \eqref{criteq} in terms of a power series, Athreya \cite{Ath12b} gave an expression in terms of sums of exponential random variables, and O'Connell \cite{OCo95} (and more recently, Harris, Johnston and Roberts \cite{HJR17}) obtained \eqref{criteq} as written in the more general near-critical setting. We refer the reader to \cite[Section 3]{HJR17} for further discussion.

\subsection{The case $k=2$, subcritical}

Finally, we look at the subcritical case $m < 1$. On the overwhelmingly rare event that a subcritical tree manages to survive until a large time $T$, the law of the number of particles alive conditioned on survival converges to a quasi-stationary limit \cite[Section III.7]{AN72}. By this, we mean that there exist non-negative numbers $\{c_j : j \geq 1\}$ satisfying $\sum_{j \geq 1} c_j = 1$ such that
\begin{align} \label{quasi}
\lim_{T \to \infty} \P( N_T = j | N_T > 0 ) = c_j. 
\end{align}
Lambert showed in \cite{Lam03} (and also Athreya in \cite{Ath12b}) that conditioned on $\{N_T \geq 2\}$, the difference $\upsilon^{L,T} := T - \tau^{L,T}$ converges in distribution to a $[0,\infty)$-valued random variable $\bar{\upsilon}^L$ as $T \to \infty$. Lambert also gave an implicit formula for the distribution of the limit variable $\bar{\upsilon}^L$, which Le \cite{Le14} inverted to obtain
\begin{align} \label{MRCASub}
\P( \bar{\upsilon}^L < t  ) = \frac{1}{ 1 - c_1}  \int_0^1  (1-s) \frac{ F_t''(s)}{ F_t'(s) } C'(s) ds,
\end{align}
where $C(s) := \sum_{ j \geq 1} c_j s^j$ is the generating function of the quasi-stationary limit. By replacing $t$ with $T-t$ in \eqref{lam_eq}, and using the fact (due to \eqref{quasi}) that
\begin{align} \label{Subconv2}
\lim_{T \to \infty} \frac{ F_T'(s) }{ \P( N_T \geq 2 ) } = \lim_{T \to \infty} \frac{ \P(N_T \geq 1)}{ \P( N_T \geq 2) } \E[ N_T s^{ N_T - 1 } | N_T \geq 1 ] = \frac{1}{ 1 - c_1} C'(s),
\end{align}
it is straightforward to sketch a proof of \eqref{MRCASub}.


\section{Main results}

\subsection{Overview of results}
Let us now give a brief overview of our main results, which will be stated formally in the sequel. Our results for general $k$ run analagously to Section 2 -- first we provide integral formulas for the law of $(\pi^{k,L,T}_t)_{t \in [0,T]}$ for fixed (finite) $T$, then we study the $T \to \infty$ asymptotics of these integrals in the supercritical, critical, and subcritical cases.\\

The fixed-$T$ results, Theorem \ref{MT FDD}, Theorem \ref{MT Split}, and  Theorem \ref{mixture markov}, characterise the law of $(\pi^{k,L,T}_t)_{t \in [0,T]}$ in three different ways, 
first in terms of its finite dimensional distributions, second in terms of its random splitting times, and third as a mixture of Markov processes. In all cases, explicit formulas are obtained, each in the form of an integral equation involving various generating functions associated with the process.\\

 For instance, the case $n=1$ of Theorem \ref{MT FDD} gives the one-dimensional distributions of $(\pi^{k,L,T}_t)_{t \in [0,T]}$. Namely, for any partiton $\gamma$ of $\{1,\ldots,k\}$,
\[ \P( \pi^{k,L,T}_t = \gamma, N_T \geq k ) = \int_0^1 \frac{(1-s)^{k-1}}{(k-1)!} F_t^{|\gamma|} \left( F_{T-t}(s) \right) \prod_{ \Gamma \in \gamma } F_{T-t}^{|\Gamma|}(s) ds, \]
where $F_t^j$ denotes the $j^{\text{th}}$ derivative of $F_t(s)$ with respect to $s$. \\

We are able to understand the combinatorial nature of the products in these integral formulas by relating them to the \emph{Fa\`a di Bruno formula} \cite{Joh02}, which states that for $k$-times differentiable $f$ and $g$, 
\begin{align} \label{FdB}
 (f \circ g )^k = \sum_{ \gamma \in \Pi^k } \left( f^{|\gamma|} \circ g \right) ~\prod_{ \Gamma \in \gamma } g^{|\Gamma|},
\end{align}
where $\Pi^k$ is the set of partitions of $\{1,\ldots,k\}$, $|\gamma|$ is the number of blocks of a partition $\gamma$ and $|\Gamma|$ are the block sizes, and $h^j$ denotes the $j^{th}$ derivative of $h$.
It transpires that there is a class of Markov processes whose finite dimensional distributions may be given in terms of a generalisation of the Fa\`a di Bruno for the semigroup $(F_t(s))_{t \geq 0}$. Theorem \ref{mixture markov} states that $(\pi^{k,L,T}_t)_{t \in [0,T]}$ is a random mixture of these processes.\\
\vspace{6mm}

We then send the picking time $T \to \infty$, and study the asymptotic behaviour of the process $(\pi^{k,L,T}_t)_{t \in [0,T]}$. As in the case $k=2$ discussed in Section 2, we will see analogous differences in the asymptotic behaviour of $(\pi^{k,L,T}_t)_{t \in [0,T]}$ depending on the mean $m := f'(1)$ of the offspring distribution. Under certain conditions, and in each case conditioned on $\{N_T \geq k\}$, we have the following as $T \to \infty$:
\begin{itemize}
\item In the supercritical case $m > 1$, Theorem \ref{MTSuper} states that we have the distributional convergence
\begin{align}
(\pi_t^{k,L,T})_{t \in [0,T]} \to ( \bar{\pi}^{k,L}_{t})_{t \in [0,\infty)},
\end{align}
for a limit process $ ( \bar{\pi}^{k,L}_{t})_{t \in [0,\infty)}$ depending on the law of $L$. We will characterise the law of the limit process $(\bar{\pi}^{k,L}_t)_{t \in [0, \infty)}$ in terms of integral formulas involving the Laplace transform of the martingale limit.
\item In the critical case $m = 1$, Theorem \ref{MTCrit} states that there exists a universal stochastic process $(\bar{\pi}^{k,\mathsf{crit}}_t)_{t \in [0,1]}$ such that
\begin{align*}
(\pi^{k,L,T}_{tT})_{t \in [0,1]}  \to (\bar{\pi}^{k,\mathsf{crit}}_t)_{t \in [0,1]},
\end{align*}
for every critical offspring distribution with finite variance. This result is not new, and was covered in detail (and in greater generality) in Harris, Johnston and Roberts \cite{HJR17}.
\item In the subcritical case $m < 1$, conditioned on survival until a large time $T$, the common ancestors of a sample of $k$ particles chosen at $T$ existed near the end of the time interval $[0,T]$. Here it makes more sense to consider $(\rho^{k,L,T}_t)_{t \in [0,T]}$ -- the right-continuous modification of $(\pi^{k,L,T}_{T-t})_{t \in [0,T]}$. Our subcritical result, Theorem \ref{MTSub}, states that 
\begin{align}
(\rho^{k,L,T}_{t})_{ t \in [0,T] } \to (\bar{\rho}^{k,L}_{t})_{ t \in [0, \infty)},
\end{align}
for a limit process $(\bar{\rho}^{k,L}_{t})_{t \in [0,\infty)}$ depending on the law of $L$. 
We will characterise the law of the limit process $(\bar{\rho}^{k,L}_{t})_{t \in [0,\infty)}$ in terms of integral formula involving the generating function of the quasi-stationary limit.
\end{itemize}
\vspace{7mm}

Finally, we consider the relationships between the processes $(\pi^{k,L,T}_t)_{t \in [0,T]}$ for different values of $k$. Theorem \ref{MT projection} states that conditioned on the event $\{N_T \geq k + j\}$, the process obtained by projecting $(\pi^{k+j,L,T})_{t \in [0,T]}$ onto $\{1,\ldots,k\}$ has the same law as $(\pi^{k,L,T}_t)_{t \in [0,T]}$. Corollary \ref{projectivity} states that the limiting processes $(\bar{\pi}^{k,L}_t)_{t \geq 0 }$ and $(\bar{\pi}^{k,\mathsf{crit}}_t)_{t \in [0,1]}$ appearing in Theorem \ref{MTSuper} and Theorem \ref{MTCrit} also satisfy a projectivity property.
\subsection{Definitions}\label{sec:definitions}
Before stating the main results in full, we need to introduce some more notation and definitions.
We start by giving a brief formal description of the continuous time Galton-Watson tree. Let $L$ be a $\{0,1,2,\ldots\}$-valued random variable and let $f(s) := \E[s^L]$ be its generating function. 
Under the probability measure $\P$, we start at time $0$ with one particle which we call $\varnothing$. 
The particle $\varnothing$ lives for a unit-mean and exponentially distributed
length of time $\tau_{\varnothing}$ until it dies, and is replaced by a random number of offspring with labels $1,2,\ldots,L_{\varnothing}$, where
$L_{\varnothing}$ is distributed like $L$ and is independent of 
$\tau_{\varnothing}$. 
These offspring then independently repeat this behaviour. 
That is, for each $u$ born at some time, $u$ lives a length of time $\tau_u$ 
distributed like $\tau_{\varnothing}$ and at death is replaced by offspring with labels $u1,u2,\ldots,uL_u$, where $L_u$ is distributed like $L$. 
Here, $\tau_u$ and $L_u$ are independent of each other and of the past. 
We write $\mathcal{N}_t$ for the set of particles alive at time $t$, $N_t = |\mathcal{N}_t|$ for the number alive at $t$, and let $F_t(s) := \E[s^{N_t}]$. We remark that $F_t(s)$ enjoys the semigroup property $F_{t_1} \circ F_{t_2} = F_{t_1 + t_2}$. \\

For $u \neq v$, we write $u < v$ if $u$ is an ancestor of $v$ (or equivalently, $v$ is a descendent of $u$) and $u \leq v$ if $u < v$ or $u = v$. Throughout we will use the terminology \emph{ancestor} and \emph{descendent} weakly, so that $u$ is both an ancestor and a descendent of itself.\\

A \emph{partition} $\gamma $ of a non-empty set $A$ 
is a collection of disjoint non-empty subsets of $A$, or \emph{blocks}, whose union is $A$. We write $| \gamma|$ for the number of blocks in $\gamma$, and for a block $\Gamma \in \gamma$, we write $|\Gamma|$ for the number of elements in $\Gamma$.
We write $\Pi^A$ for the collection of partitions of $A$, and write $\Pi^k := \Pi^{\{1,\ldots,k\} }$. If $B$ is a non-empty subset of $A$, and $\alpha$ is a partition of $A$, we write $\alpha^B$ (or $\alpha|^B$ when there are other superscripts present) for the projection of $\alpha$ onto $B$:
\begin{align*}
\alpha^B := \{ A' \cap B ~\text{non-empty} : A' \in \alpha \}.
\end{align*}
When projecting a partition $\alpha$ onto the set $\{1,\ldots,k\}$, we will write $\alpha^k$ (or $\alpha|^k$) in place of $\alpha^{\{1,\ldots,k\}}$. \\

For partitions $\alpha, \beta$, we say $\alpha$ can break into $\beta$, 
written $\alpha \prec \beta$ (or $\beta \succ \alpha$), if each block of $\alpha$ is a union of blocks in $\beta$. 
For example, $\big\{ \{1,2,4\}, \{3\}\big\} \prec \big\{ \{1\}, \{2,4\}, \{3\}\big\}$. 
An \emph{$n$-chain} (or just \emph{chain}) of partitions is a sequence of partitions $\boldsymbol\gamma = (\gamma_1, \ldots, \gamma_n)$ with the property that $\gamma_i \prec \gamma_{i+1}$ for every $i$. Let $\Pi^A_n$ denote the set of $n$-chains of partitions of $A$, and for $k \geq 1$ write $\Pi^k_n := \Pi_n^{ \{1,\ldots,k\}}$. Using the conventions $\gamma_0 = \big\{ \{1,\ldots,k\} \big\}$ and $\gamma_{n+1} = \big\{ \{1\}, \{2\}, \ldots, \{k \} \big\}$, for each $0 \leq i \leq n$, every block $\Gamma \in \gamma_i$ is the union of $b_i(\Gamma) \geq 1$ blocks of $\gamma_{i+1}$. Ordering blocks by their least element, we call the doubly indexed array $( b_i(\Gamma) : 0 \leq i \leq n, \Gamma \in \gamma_i )$ the \emph{fragmentation numbers} associated with the chain $\boldsymbol\gamma$, and $(b_i(\Gamma) : \Gamma \in \gamma_i)$ the \emph{fragmentation numbers at the level $i$}.\\

We will also use the terminology \emph{chain} for sequence of partitions satisfying $\gamma_i \succ \gamma_{i+1}$. Adopting the convention $\gamma_0 = \{ \{1\}, \ldots, \{k\} \}$ and $\gamma_{n+1}= \{ \{1,\ldots,k\} \}$, for each $0 \leq i \leq n$, each block $\Gamma \in \gamma_i$ is the union of $m_i(\Gamma)$ blocks of $\gamma_{i-1}$. We call the array $( m_i(\Gamma) : i \leq n, \Gamma \in \gamma_i )$ the \emph{merger numbers}.\\

A \emph{mesh} $(t_i)_{i \leq n}$ of a time interval $[0,T]$ is a collection of times $0 <  t_1 < \ldots < t_n < T$. Given a mesh $(t_i)_{i \leq n}$, we set $\Delta t_i := t_{i+1} - t_i$ (employing the convention $t_0 = 0, t_{n+1} = T$).\\

Whenever $h$ is a function, $h^j$ or $h^j(s)$ will refer to the $j^\text{\tiny{th}}$-derivative of the function (and $h(s)^j$ for the $j^\text{\tiny{th}}$ exponent). In particular, we will write $F^j_t(s)$ for the $j^\text{\tiny{th}}$-derivative of $F_t(s)$ with respect to $s$. 

\vspace{8mm}

Recall that under $\P$ we have a continuous-time Galton-Watson tree starting with one initial particle, branching at rate $1$ and with offspring numbers distributed like $L$. Additionally under $\P$, and on the event $\{N_T \geq k\}$, pick $k$ distinct particles $U_1,\ldots,U_k$ uniformly from those alive at time $T$. 
For each time $t \in [0,T]$, we define the equivalence relation 
$i \sim_{t} j$ if and only if $U_i$ and $U_j$ share a common ancestor alive at time $t$. 
We let $\pi^{k,L,T}_t$ denote the random partition of $\{1,\ldots,k\}$ corresponding to this equivalence relation. The resulting process $(\pi^{k,L,T}_t)_{t \in [0,T]}$, defined on the event $\{N_T \geq k\}$, is a right-continuous partition-valued stochastic process satisfying
\begin{align*}
\pi^{k,L,T}_0 = \Big\{ \{1,2,\ldots.,k\} \Big\} ,~~ \pi^{k,L,T}_T = \Big\{ \{1\}, \{2\}, \ldots, \{k\} \Big\}.
\end{align*}
Furthermore, $(\pi^{k,L,T}_t)_{t \in [0,T]}$ is a \emph{fragmentation process}, in the sense that blocks break as time passes:
\begin{align*}
t_1 < t_2 \implies \pi^{k,L,T}_{t_1} \prec \pi^{k,L,T}_{t_2}.
\end{align*}

Let $(\rho^{k,L,T}_t)_{t \in [0,T]}$ be the right-continuous modification of $(\pi^{k,L,T}_{T-t})_{t \in [0,T]}$. Then $(\rho^{k,L,T}_t)_{t \in [0,T]}$ is a \emph{coalescent process}, in the sense that blocks merge together as time passes:
\begin{align*}
t_1 < t_2 \implies \rho^{k,L,T}_{t_1} \succ \rho^{k,L,T}_{t_2}.
\end{align*}

We call the discontinuities $\tau_1 < \ldots < \tau_n$ of $(\pi^{k,L,T}_t)_{t \in [0,T]}$ \emph{split times}, since these times correspond to a block splitting into two or more blocks.\\

The event $\{ (\pi^{k,L,T}_t)_{t \in [0,T]}\text{ is binary} \}$ refers to the event that at every split time, a block splits into exactly two blocks. Note that
\begin{align*}
\{ (\pi^{k,L,T}_t)_{t \in [0,T]}\text{ is binary} \} = \{ (\pi^{k,L,T}_t)_{t \in [0,T]} \text{ has $k-1$ split times} \}.
\end{align*}
When the offspring generating function is of the form $f(s) = \alpha + \gamma s + \beta s^2$, no particle in the tree has more than two offspring upon death, and we call the underlying Galton-Watson tree a \emph{birth-death process}. It follows from that when the underlying tree is a birth-death process
\begin{align*}
\P \left(  (\pi^{k,L,T}_t)_{t \in [0,T]}\text{ is binary}~ \big|~ N_T \geq k \right) = 1.
\end{align*}
Whenever the tree is not a birth-death process, there is a positive probability that before time $T$ some particle in the underlying Galton-Watson tree is replaced by three or more offspring upon death, and hence
\begin{align*}
\P \left( (\pi^{k,L,T}_t)_{t \in [0,T]}\text{ is binary}~   \big| ~N_T \geq k \right) < 1
\end{align*}
for every $k \geq 3$. \\

\subsection{Hypotheses}\label{sec:hypotheses}
We need to ensure that there actually \textit{are} at least $k$ particles alive at time $T$ with positive probability, and that we can choose uniformly from them.
To be more precise, we must ensure that
both $\P(N_T \geq k) > 0$, and $\P(N_T < \infty) = 1$. 
The inequality  $\P(N_T \geq k) > 0$ is guaranteed to hold by virtue of our first hypothesis, which states that
\begin{align} \label{Non-triv}
f''(1) > 0.
\end{align}

In addition to \eqref{Non-triv}, we insist that the following non-explosion hypothesis holds:
\begin{align} \label{Non-exp}
\int_{1-\epsilon}^1 \frac{ ds }{ | f(s) - s |} = \infty,
\qquad 
\forall \epsilon \in (0,1).
\end{align}
This condition \eqref{Non-exp} is equivalent to our second requirement that $\P( N_t < \infty ) = 1$ for $t$, and holds whenever $f'(1) < \infty$ \cite[Chapter II, Theorem 9.1]{Har63}. 
We emphasize that both hypotheses \eqref{Non-triv} and \eqref{Non-exp} are in force in the remainder of this paper.\\
\bigskip

We are now ready to state our main results, which we split into three sections.
The results in Section \ref{SecLaw} concern fixed and finite $T$.
The results in Section \ref{SecAsy} concern the asymptotic regime in which $T$ is sent to $\infty$.
The results in Section \ref{SecProjectivity} concern the projectivity of the partition processes.

\subsection{Fixed-$T$ results} \label{SecLaw}
The results in this section describe the law of the process $(\pi^{k,L,T}_t)_{t \in [0,T]}$ for fixed times $T$ in terms of the generating functions $F_t(s) = \E[s^{N_t}]$ and $f(s) = \E[s^{L} ]$. \\

Our first fixed-$T$ result, Theorem \ref{MT FDD}, is a generalisation of Lambert's equation \eqref{lam_eq}, giving the finite dimensional distributions of the stochastic process $(\pi^{k,L,T}_t)_{t \in [0,T]}$.

\begin{thm}\label{MT FDD}
For any mesh $(t_i)_{i \leq n}$, and any chain of partitions $\boldsymbol\gamma = (\gamma_1, \ldots, \gamma_n)$ of $\{1,\ldots,k\}$, 
\begin{align} \label{EqFDD}
\P( \pi^{k,L,T}_{t_1} = \gamma_1, \ldots, \pi^{k,L,T}_{t_n} = \gamma_n, ~ N_T \geq k) = \int_0^1  \frac{(1-s)^{k-1}  }{(k-1)! } \prod_{i = 0}^{ n} \prod_{\Gamma \in \gamma_i } F_{\Delta t_i }^{b_i(\Gamma) } \left( F_{T- t_{i+1}}(s) \right) ds,
\end{align}
where $\Delta t_i = t_{i+1} - t_i$.
\end{thm}
Our next result, Theorem \ref{MT Split}, characterises the law of $(\pi^{k,L,T}_t)_{t \in [0,T]}$ in terms of its split times. To this end, let $\boldsymbol{\eta} = (\eta_0,\ldots,\eta_{n})$ be a chain of partitions such that $\eta_0 = \{ \{1,\ldots,k\} \}$, $\eta_{n} = \{ \{1\},\ldots,\{k\}\}$. We say $\boldsymbol\eta$ is \emph{maximal} if $\eta_{i}$ is obtained from $\eta_{i-1}$ by breaking precisely one block of $\eta_{i-1}$ into $q_i \geq 2$ blocks in $\eta_{i}$. \\

Roughly speaking, given a maximal chain $\boldsymbol\eta = (\eta_0,\ldots,\eta_{n})$, the following theorem gives the joint density of the $n$ times that the process `jumps' from the value $\eta_{i-1}$ to $\eta_i$, characterising the joint law of the split times of $(\pi^{k,L,T}_t)_{t \in [0,T]}$.

\begin{thm} \label{MT Split}
Let $\boldsymbol\eta = (\eta_0,\ldots,\eta_n)$ be a maximal chain, and let $a_1 < b_1 < a_2 < b_2 < \ldots a_n < b_n$. Then
\begin{align} \label{EqSplit}
& \P( \pi^{k,L,T}_{a_i} = \eta_{i-1}, \pi^{k,L,T}_{b_i} = \eta_{i} ~\forall~ i = 1,\ldots,n, ~ N_T \geq k )\\ 
&= \int_{a_1}^{b_1} \ldots \int_{a_n}^{b_n} du_1 \ldots du_n \int_0^1 \frac{ (1-s)^{k-1}}{ (k-1)!  } F_T'(s)\prod_{ i=1}^n f^{q_i}( F_{T-u_i} (s)) F_{T-u_i}'(s)^{q_i - 1}  ds,
\end{align}
where $q_i = 1 + |\eta_i| - |\eta_{i-1}|$.
\end{thm}

The following lemma is a generalisation of the Fa\`a di Bruno formula \eqref{FdB}, shedding light on the products occuring in the integral in \eqref{EqFDD}.
\begin{lem} \label{fdb j}
Let $g_0, \ldots, g_n$ be $k$-times differentiable. Then
\begin{align} \label{fdb gen}
 (g_0  \circ g_1 \circ \ldots \circ g_n )^k = \sum_{ \boldsymbol\gamma \in \Pi^k_n } \prod_{i = 0}^n \prod_{ \Gamma \in \gamma_i } g_i^{ b_i(\Gamma) } \circ g_{i+1} \circ \ldots \circ g_n.
\end{align}
In particular, for any semigroup $(F_t)_{t \geq 0}$ of $k$-times differentiable functions, and any mesh $(t_i)_{i \leq n}$ of $[0,T]$, we have
\begin{align} \label{semigroup FdB}
 F_T^k(s) = \sum_{ \boldsymbol\gamma \in \Pi^k_n } \prod_{i = 0}^n \prod_{ \Gamma \in \gamma_i } F_{\Delta t_i }^{ b_i(\Gamma)} \left( F_{T - t_{i+1}}(s)\right).
\end{align}
\end{lem}
The identity \eqref{semigroup FdB} is used to interpret $(\pi^{k,L,T}_t)_{t \in [0,T]}$ as a mixture of Markov processes as follows.\\

Recall that if $B$ is a subset of $A$ and $\alpha$ is a partition of $A$, $\alpha^B$ is the projection of $\alpha$ onto $B$. We say a $\Pi^k$-valued (time-inhomogeneous) Markov fragmentation process $(\tilde{\pi}_t)_{t \in [0,T]}$ has the \emph{independent blocks property} if given $\{\tpi_{t_0 } = \gamma\}$, the stochastic processes $\left\{ (\tpi_t^\Gamma)_{t \in [t_0,T]} : \Gamma \in \gamma \right\}$ are conditionally independent. Any such process under a law $\P$ is characterised by the quantities \[\P( \pi^\Gamma_{t_2}  = \delta | \pi_{t_1} = \gamma ),\] 
where $t_1 < t_2 < T$, $\gamma$ is a partition of $\{1,\ldots,k\}$, $\Gamma$ is a block of $\gamma$, and $\delta$ is a partition of the block $\Gamma$. \\

In Section \ref{SecMixture} we show using \eqref{semigroup FdB} that there exists a Markov process $(\tpi_t)_{t \in [0,T]}$ starting from $\tpi_0 = \{ \{ 1,\ldots,k\} \}$ under a probability law $\mathbb{R}^{k,L,T}_s$ with the independent blocks property and transition density 
\begin{align} \label{tpi rates}
\mathbb{R}^{k,L,T}_s( \tpi^\Gamma_{t_2} = \delta | \tpi_{t_1} = \gamma) := \frac{ F_{t_2-t_1}^{|\delta|} ( F_{T-t_2}(s)) \prod_{ \Delta \in \delta } F_{T-t_2}^{|\Delta|}(s) }{ F_{T-t_1}^{|\Gamma|}(s) }, ~~~ t_2 \geq t_1.
\end{align}

Our final fixed-$T$ result, Theorem \ref{mixture markov}, states that the process $(\pi^{k,L,T}_t)_{t \in [0,T]}$ can be constructed as a random mixture $m^{k,L,T}$ of processes with laws given by $\{ \mathbb{R}^{k,L,T}_s : s \in [0,1] \}$, where the mixture measure is given by
\begin{align} \label{mixture measure}
m^{k,L,T}(ds) :=  \frac{ (1-s)^{k-1} F_T^k(s) }{ (k-1)! \P( N_T \geq k ) } .
\end{align}

\begin{thm} \label{mixture markov}
The conditional law of $(\pi^{k,L,T}_t)_{t \in [0,T]}$ on the event $\{N_T \geq k \}$ is given by 
\begin{align*}
\P ( \pi^{k,L,T}_{t_1} = \gamma_1, \ldots, \pi^{k,L,T}_{t_n} = \gamma_n ~ | ~ N_T \geq k ) = \int_0^1 m^{k,L,T}(ds) \mathbb{R}^{k,L,T}_s ( \tpi_{t_1} = \gamma_1, \ldots, \tpi_{t_n} = \gamma_n  ).
\end{align*}
In particular, $m^{k,L,T}(ds)$ is a probability measure and $(\pi^{k,L,T}_t)_{t \in [0,T]}$ is a mixture of Markov processes with the independent blocks property.
\end{thm}

\subsection{Asymptotic-$T$ results} \label{SecAsy}

We now move on to results concerning the asymptotic behaviour of $(\pi^{k,L,T}_t)_{t \in [0,T]}$ as $T \to \infty$. Below we say a collection of partition-valued stochastic processes $\{(\pi^T_t)_{t \geq 0} : T > 0\}$ converge in distribution to a stochastic process $(\bar{\pi}_t)_{t \geq 0 }$ as $T \to \infty$ if the finite dimensional distributions converge:
\begin{align*}
\lim_{T \to \infty} \P( \pi^T_{t_1} = \gamma_1,\ldots, \pi^T_{t_n} = \gamma_n) = \P( \bar{\pi}_{t_1} = \gamma_1, \ldots, \bar{\pi}_{t_n} = \gamma_n).
\end{align*}

First we will consider the supercritical case $m>1$. Given the stochastic process $(\pi^{k,L,T}_t)_{t \in [0,T]}$ defined on $[0,T]$, we define its extension $(\pi^{k,L,T}_t)_{t \geq 0}$ to all of $[0,\infty)$ by setting 
\begin{align*}
\pi^{k,L,T}_t = \big\{ \{1\}, \ldots, \{ k \} \big\}
\end{align*}
whenever $t > T$.

\begin{thm} \label{MTSuper}
Let $m > 1$ and $\E[ L \log_+ L] < \infty$. Then as $T \to \infty$, conditioned on $\{N_T \geq k\}$ the process $(\pi^{k,L,T}_t)_{t \geq 0}$ converges in distribution to a stochastic process $(\bar{\pi}^{k,L}_t)_{t \geq 0 }$ with finite dimensional distributions given by 
\begin{align} \label{super}
& \P( \bar{\pi}^{k,L}_{t_1} = \gamma_1, \ldots, \bar{\pi}^{k,L}_{t_n}  = \gamma_n) \nonumber
\\ &=\frac{(-1)^k e^{ - k(m-1)t_n} }{ 1 - \varphi(\infty)} \int_0^\infty \frac{ v^{k-1}}{ (k-1)!} \prod_{i =0}^{n-1} \prod_{ \Gamma \in \gamma_i} F_{ \Delta t_i}^{b_i(\Gamma)} \left( \varphi( e^{ - (m-1)t_{i+1}} v ) \right) \prod_{ \Gamma \in \gamma_n} \varphi^{ | \Gamma|  }( e^{ - (m-1)t_n } v ) dv, \end{align}
where $\varphi(v) := \E [ e^{ - v W_\infty} ]$ is the Laplace transform of the martingale limit $W_\infty := \lim_{T \to \infty} N_T e^{- (m-1)T}.$
\end{thm}
\bigskip

Next we consider the critical case $m=1$.
\begin{thm} \label{MTCrit}
There exists a universal stochastic process $(\bar{\pi}^{k,\mathsf{crit}}_t)_{t \in [0,1]}$ such that for any tree with $m = 1$ and $f''(1) < \infty$, the process $(\pi^{k,L,T}_{T t})_{t \in [0,1]}$ conditioned on $\{N_T \geq k\}$ converges in distribution to $(\bar{\pi}^{k,\mathsf{crit}}_t)_{t \in [0,1]}$ as $T \to \infty$. Moreover, the finite dimensional distributions of $(\bar{\pi}^{k,\mathsf{crit}}_t)_{t \in [0,1]}$ are given by
\begin{align} \label{crit}
& \P( \bar{\pi}^{k,\mathsf{crit}}_{ t_1} = \gamma_1, \ldots, \bar{\pi}^{k,\mathsf{crit}}_{t_n} = \gamma_n ) \\
& =\prod_{i=0}^n \prod_{ \Gamma \in \gamma_i}b_i(\Gamma)! \int_0^\infty \frac{ \theta^{k-1}}{(k-1)!} \prod_{i=0}^n (\Delta t_i)^{|\gamma_{i+1}| - | \gamma_i| } \left( \frac{ 1 + (1-t_{i+1})\theta}{ 1 + (1 - t_i)\theta } \right)^{|\gamma_{i+1}|} d\theta. \label{fdd crit}
\end{align}

\end{thm}

\vspace{8mm}

Theorem \ref{MTCrit} has already appeared (albeit from a different perspective) in Harris, Johnston and Roberts \cite{HJR17}, who show that the genealogical tree corresponding to $(\bar{\pi}^{k,\mathsf{crit}}_t)_{t \in [0,1]}$ is binary, and the $k-1$ split times of $(\bar{\pi}^{k,\mathsf{crit}}_t)_{t \in [0,1]}$ have joint probability density function 
\begin{align} \label{HJR density}
P(u_1, \ldots, u_{k-1} ) = k \int_0^\infty \frac{ \theta^{k-1} }{ (1 + \theta)^2} \prod_{i=1}^{k-1} \frac{1}{ (1 + \theta(1 - u_i) )^2 } d\theta. 
\end{align}
It is possible to derive \eqref{HJR density} from our formula \eqref{fdd crit} by letting $t_{2i-1} = u_i, t_{2i} = u_i + h_i$ for $i = 1,\ldots,k-1$, and sending every $h_i \downarrow 0$. The resulting discrepancy by a factor $\frac{k!(k-1)!}{2^k}$ is a matter of counting tree topologies: there are $\frac{k!(k-1)!}{2^k}$ ranked binary trees with $k$ labelled leaves and $(k-1)$ ranked internal nodes \cite{Murtagh}.\\

\bigskip

Finally, we look at the subcritical case $m<1$. Recall that $(\rho^{k,L,T}_t)_{t \in [0,T]}$ is the right-continuous modification of $(\pi^{k,L,T}_{T-t})_{t \in [0,T]}$. We define the extension $(\rho^{k,L,T}_t)_{t \geq 0
}$ to all of $[0,\infty)$ by setting
\begin{align*}
\rho^{k,L,T}_t := \big\{ \{1, \ldots, k\} \big\}
\end{align*}
whenever $t > T$. 
\begin{thm} \label{MTSub}
Let $m < 1$ and $\E[ L \log_+ L ]< \infty$. Then as $T \to \infty$, conditioned on $\{N_T \geq k\}$, the process $(\rho^{k,L,T}_t)_{t \geq 0}$ converges in distribution to a stochastic process $(\bar{\rho}^{k,L}_t)_{t \geq 0 }$ with finite dimensional distributions given by 

\begin{align} 
& \P( \bar{\rho}^{k,L}_{t_1} = \gamma_1, \ldots, \bar{\rho}^{k,L}_{t_n} = \gamma_n) \nonumber \\
&= \frac{ e^{ - (m-1)t_n} }{ 1 - \sum_{j = 1}^{k-1} c_j } \int_0^1 \frac{(1-s)^{k-1}}{ (k-1)! }  C^{ | \gamma_n| }( F_{t_n}(s) )  \prod_{ i = 1 }^n \prod_{ \Gamma \in \gamma_i } F_{ \Delta t_{j-1}}^{ m_j(\Gamma) } ( F_{t_{j-1}}(s)) ~ds, \label{sub}
\end{align}
where $c_j := \lim_{T \to \infty} \P( N_T = j | N_T > 0)$ and $C(s) := \sum_{j \geq 1} c_j s^j$. 

\end{thm}

\subsection{Projectivity}\label{SecProjectivity}
Let $k$ and $j$ be positive integers. Recall that for $\gamma \in \Pi^{k+j}$, $\gamma|^k \in \Pi^k$ is the projection of $\gamma$ onto $\{1,\ldots,k\}$. Writing $(\pi^k_t)_{t \in [0,T]} := (\pi^{k,L,T}_t)_{t \in [0,T]}$ for the remainder of this section, it is natural to expect from the definition of $(\pi^k_t)_{t \in [0,T]}$ that the projected process $\left(\pi^{k+j}_t\big|^k\right)_{t \in [0,T]}$ is closely related to  $(\pi^k_t)_{t \in [0,T]}$ -- a property we call \emph{projectivity}. The following generalisation of Theorem \ref{MT FDD} clarifies this connection. 

\begin{thm} \label{MT projection}
On the event $\{N_T \geq k+j\}$, the processes $\left(\pi^{k+j}_t\big|^k\right)_{t \in [0,T]}$ and $(\pi^k_t)_{t \in [0,T]}$ are identical in law, and have finite dimensional distributions given by
\begin{align} \label{MT projection eq}
\P(  \pi^{k}_{t_1} = \gamma_1, \ldots,  \pi^{k}_{t_n} = \gamma_n, ~ N_T \geq k+j  ) 
= \int_0^1 \frac{ (1-s)^{k+j-1}}{ (k+j-1)! } \frac{ \partial^j}{ \partial s^j} \left(  \prod_{i=0}^n \prod_{\Gamma \in \gamma_i} F_{\Delta t_i}^{b_i(\Gamma) }\left( F_{T-t_{i+1}}(s) \right) \right) ds   .
\end{align}
\end{thm}
Theorem \ref{MT projection} has two immediate corollaries. The first of these, Corollary \ref{projectivity}, states that when the underlying Galton-Watson tree is either supercritical or critical, the discrepancy in the conditioning disappears in the limit.
\begin{cor} \label{projectivity}
Let $(\bar{\pi}^{k,L}_t)_{t \geq 0}$ be defined as in Theorem \ref{MTSuper}, and let $(\pi^{k,\mathsf{crit}}_t)_{t \in [0,1]}$ be defined as in Theorem \ref{MTCrit}. Then the processes $\left(\bar{\pi}^{k+j,L}_t|^k \right)_{ t \geq 0 }$ and $(\bar{\pi}^{k,L}_t)_{t \geq 0}$ are identical in law, and the processes $\left(\bar{\pi}^{k+j,\mathsf{crit}}_t\big|^k \right)_{t \in [0,1]}$ and $(\bar{\pi}^{k,\mathsf{crit}}_t)_{t \in [0,1]}$ are identical in law.
\end{cor}
\bp
In both the supercritical and critical cases, the proof is a consequence of Theorem \ref{MT projection} and the fact that
\begin{align} \label{proof a}
\lim_{T \to \infty} \P( N_T \geq k+j ~|~ N_T \geq k) = 1.
\end{align}
See \cite[Chapter III]{AN72} for details on \eqref{proof a}. \ep

We emphasise that no exact analogue of Corollary \ref{projectivity} holds in the subcritical case, since due to \eqref{quasi} we have
\begin{align*}
\lim_{T \to \infty} \P( N_T \geq k+j ~|~ N_T \geq k)  = \frac{ \sum_{i \geq k+j} c_j }{  \sum_{i \geq k} c_j } < 1.
\end{align*}

The second corollary of Theorem \ref{MT projection} gives the finite dimensional distributions of $(\pi^k_t)_{t \in [0,T]}$ on the event $\{N_T = k+j \}$.

\begin{cor}
For any $j \geq 0$, 
\begin{align*}
\P(  \pi^{k}_{t_1} = \gamma_1, \ldots,  \pi^{k}_{t_n} = \gamma_n, ~ N_T = k+j  ) = \frac{1}{(k+j)!} \frac{ \partial^{j}}{ \partial s^{j}} \left \{ \prod_{i=0}^n \prod_{\Gamma \in \gamma_i} F_{\Delta t_i}^{b_i(\Gamma) }( F_{T-t_{i+1}}(s) ) \right\} \Bigg|_{s=0}.
\end{align*}
\end{cor}
\bp
Let $S(k,j)$ denote the integral in \eqref{MT projection eq}. Then 
\begin{align*}
\P(  \pi^{k}_{t_1} = \gamma_1, \ldots,  \pi^{k}_{t_n} = \gamma_n, ~ N_T =  k+j) &= S(k,j) - S(k,j+1).
\end{align*}
Now integrate by parts.
\ep

\subsection{Further discussion of related literature}
After this paper first appeared on arXiv, Fa\`a di Bruno's formula has made appearances in recent work by Vatutin and his coauthors on the genealogical structure of branching processes \cite{LV,VHJ}. In \cite{LV}, Liu and Vatutin develop results on the reduced processes associated with critical Galton-Watson trees conditioned to have a relatively small population, and use these results to obtain formulas for the time to most recent common ancestor of the entire population. A similar problem is studied by Vatutin, Hong and Ji \cite{VHJ} in the context of critical Bellman-Harris branching processes --- the generalisation of Galton-Watson trees with non-exponential lifetimes.\\

In \cite{grosjean_huillet:coalescence}, Grosjean and Huillet examined the genealogical structure of discrete-time Galton-Watson trees, providing the following extension of Lambert's equation \eqref{lam_eq}. On the event $\{N_T \geq k\}$, let $\tau^{k,L,T}$ be the time at which $k$ distinct particles chosen uniformly at time $T$ last shared a common ancestor. Then \cite[Proposition 2.2]{grosjean_huillet:coalescence} states that
\begin{align} \label{EqKMRCA}
\P( \tau^{k,L,T} >t, N_T \geq k ) =  \int_0^1 \frac{(1-s)^{k-1}}{ (k-1)!}  \frac{ F_T'(s)}{F_{T-t}'(s)}  F^{k}_{T-t}(s) ds.
\end{align}

Note by the semigroup property $F_{t_1}(F_{t_2}(s)) = F_{t_1 + t_2}(s)$ that $
F_t'(F_{T-t}(s)) = \frac{F_T'(s)}{F_{T-t}'(s)}$, and also that by definition $\{ \tau^{k,L,T } > t \} = \big\{ \pi^{k,L,T}_t = \{ \{1, \ldots, k \} \} \big\}$. Combining these two facts, we see that \eqref{EqKMRCA} corresponds to the special case $n=1$, $\gamma_1 = \{ \{1,\ldots,k\} \}$ of Theorem \ref{MT FDD}.\\

In \cite{Le14}, Le studied the coalescent structure of continuous-time Galton-Watson trees starting with $x \geq 1$ individuals. (We remark that when $x > 1$, the random initial partition $\pi^{k,L,T}_0$ of $(\pi^{k,L,T}_t)_{t \in [0,T]}$ may have more than one block.) Le gave an implicit representation for the split times $\tau_1 < \ldots < \tau_{k-1}$ of $(\pi^{k,L,T}_t)_{t \in [0,T]}$ on the joint event
\begin{align*}
A := \left\{ \pi^{k,L,T}_0 = \{ \{1,\ldots,k\} \}, ~(\pi^{k,L,T}_t)_{t \in [0,T]}\text{ is binary}  \right\} .
\end{align*}
Namely, in a tree starting with $x \geq 1$ individuals, \cite[Theorem 4.2]{Le14} states that
\begin{align}
&\E_x[ N_T^{(k)} s^{N_T - k },~ \tau_i \in dt_i~\forall~i, A, N_T \geq k ] \nonumber \\
&= \frac{k!(k-1)!}{2^{k-1}} x F_T'(s) F_T(s)^{x-1} \prod_{i=1}^{k-1} F_{T - t_i}'(s) f''( F_{T - t_i}(s) ) dt_i  \label{LeImplicit},
\end{align}
for any $0 < t_1 < \ldots < t_{k-1} < T$, where $n^{(k)} := n(n-1)\ldots(n-k+1)$. \\

We can relate \eqref{LeImplicit} to a special case of Theorem \ref{MT Split} by calling upon an inversion formula we prove below, Lemma \ref{Betalem}, which states that when $N$ is a $\{0,1,\ldots\}$-valued random variable and $X$ is a non-negative random variable on some probability space, then
\begin{align} \label{EqBetalem1}
\int_0^1 \frac{ (1-s)^{ k - 1 } \E[  N^{(k)} s^{N- k}~ X] }{ (k-1)!  }  ds= \E[ X\ind_{N \geq k} ].
\end{align}
Without too much concern for technicalities surrounding whether $X := \ind \left \{\tau_i \in dt_i ~\forall~i, A \right\}$ constitutes a well-defined random variable, by applying \eqref{EqBetalem1} to Le's formula \eqref{LeImplicit} we obtain
\begin{align} \label{le inv}
& \P_x ( \tau_i \in dt_i ~\forall~i, A, N_T \geq k ) \nonumber \\
&=  \frac{k!}{2^{k-1}} \int_0^1 (1-s)^{k-1}  x F_T'(s) F_T(s)^{x-1} \prod_{i=1}^{k-1} F_{T - t_i}'(s) f''( F_{T - t_i}(s) ) ds.
\end{align}
Setting $x=1$ in \eqref{le inv}, and accounting for a topological factor of $\frac{k!(k-1)!}{2^{k-1}}$ (counting the number of binary trees with $k$ labelled leaves and $k-1$ ranked internal nodes \cite{Murtagh}), \eqref{le inv} corresponds the special case  of Theorem \ref{MT Split} obtained by setting
\begin{align*}
(q_1,\ldots,q_{n}) = (\underbrace{2,2,\ldots,2}_{\text{$k-1$ times}}).
\end{align*}

Finally, let us discuss Harris, Johnston and Roberts \cite{HJR17}, who looked at the genealogical structure of Galton-Watson trees in two cases. First they considered the genealogy of birth-death processes for fixed times $T$, and thereafter they studied the genealogy of trees under the near-critical scaling limit
\begin{align} \label{bee}
f'(1) = 1+ \mu/T +o(1/T), ~~~ f''(1) = \sigma^2 + o(1), ~~ T \to \infty.
\end{align}
In the birth-death case where $f(s) = \alpha + \beta s^2$, the process $(\pi^{k,L,T}_t)_{t \in [0,T]}$ is binary, and \cite[Proposition 20]{HJR17} states that conditioned on $\{N_T \geq k \}$, the joint density of the $k-1$ split times of $(\pi^{k,L,T}_t)_{t \in [0,T]}$ is given by
\begin{align} \label{HJReq}
&P(t_1,\ldots,t_{k-1}) \nonumber \\
&= \frac{k!(\beta e^{(\beta-\alpha)T}-\alpha)^k (\beta-\alpha)^{2k-1}}{(e^{(\beta-\alpha)T}-1)^{k-1} e^{(\beta-\alpha)T}}\hspace{-1.5mm} \int_0^1 (1-s)^{k-1} \prod_{j=0}^{k-1} \frac{e^{(\beta-\alpha)(T-t_j)}}{(\beta(1-s)e^{(\beta-\alpha)(T-t_j)} + \beta s - \alpha)^2} ds,
\end{align}
whenever $\alpha \neq \beta$ (with a similar formula holding when $\alpha = \beta$). It is possible to obtain \eqref{HJReq} from Theorem \ref{MT Split} of the present paper by setting $n = k-1$, $(c_1, \ldots,c_n) = (2,\ldots,2)$, and using the fact that 
\begin{align} \label{rox}
F_t(s) = \frac { \alpha(1-s)e^{ ( \beta-\alpha)t } + \beta s - \alpha }{ \beta ( 1- s)e^{ (\beta - \alpha)t }+  \beta s - \alpha },
\end{align}
where \eqref{rox} can be derived using Kolmogorov's forward equation (see for instance, \cite[Chapter III, Section 5]{AN72}).\\

As for the near-critical scaling limit \eqref{bee}, \cite[Theorem 3]{HJR17} states that conditioned on $\{N_T \geq k\}$, the process $(\pi^{k,L,T}_{Tt})_{t \in [0,1]}$ converges in distribution to a binary process $(\pi^{k,\mathsf{crit},\mu}_t)_{t \in [0,1]}$. (We have translated this result into our notation --  the process we study in Theorem \ref{MTCrit} of the present paper corresponds to the special case $\mu=0$.)  Moreover, according to \cite[Section 2.3]{HJR17}, the $k-1$ split times of the process $(\pi^{k,\mathsf{crit},\mu}_t)_{t \in [0,1]}$ have probability density function
\[f_k(t_1,\ldots,t_{k-1}) = \begin{cases} \displaystyle k(r\mu)^{k-1}(1-e^{-r\mu}) \int_0^\infty \theta^{k-1} \prod_{i=0}^{k-1} \frac{e^{r\mu(1-t_i)}}{(1+\theta(e^{r\mu(1-t_i)}-1))^2}\,\d\theta & \hbox{ if } \mu > 0\\
                                          \displaystyle k \int_0^\infty \theta^{k-1} \prod_{i=0}^{k-1} \frac{1}{(1+\theta(1-t_i))^2}\,\d\theta & \hbox{ if } \mu = 0\\
                                          \displaystyle k(-1)^k(r\mu)^{k-1}(1-e^{-r\mu})\hspace{-1.5mm} \int_0^\infty \hspace{-2mm}\theta^{k-1}\hspace{-1mm} \prod_{i=0}^{k-1} \frac{e^{r\mu(1-t_i)}}{(1-\theta(e^{r\mu(1-t_i)}-1))^2}\d\theta & \hbox{ if } \mu < 0,
\end{cases}\]
with the convention $t_0 = 0$. Harris et al also look at the topology of $(\pi^{k,\mathsf{crit},\mu}_t)_{t \in [0,1]}$, showing the tree drawn out is topologically equivalent to \emph{Kingman's coalescent} \cite{Kin82}. \\

The results in \cite{HJR17} are obtained using \emph{multiple spines}, a collection of stochastic processes 
\[\left( (\xi^1_t)_{t \geq 0},\ldots,(\xi^k_t)_{t \geq 0} \right)\]
 that `flow' through the tree forward in time, and introduced a change of measure $\Q^{k,T}$ on these spines that biases their behaviour such a way that the spines $(\xi^1_T,\ldots,\xi^k_T)$ represent a uniform sample of $k$ distinct particles at time $T$.\\

When the underlying Galton-Watson tree is a birth-death process (respectively, a near-critical tree), the subtree traced out by the spines under this change of measure $\Q^{k,T}$ is binary (respectively, asymptotically binary), and hence the law of this subtree has a tractable expression in terms of its $k-1$ split times. 
A side effect of the change of measure $\Q^{k,T}$ is that it distorts the law of the underlying Galton-Watson tree, and the main challenge in \cite{HJR17} was the inversion of this change of measure using a variant of Campbell's formula.\\

We were inspired in part by \cite{HJR17}, and there is some contentual overlap in Section \ref{SpinesSec} of the present paper and \cite[Section 4]{HJR17}, which we now discuss.
The methodology used in \cite{HJR17} is reliant on the fact that the tree traced out by the spines is binary (or equivalently $(\pi^{k,L,T}_t)_{t \in [0,T]}$ is binary), and hence has a law easily expressed in terms of the split times.  
Though we also use multiple spines in working towards our main result, Theorem \ref{MT FDD}, our method differs in order to encompass non-binary spine trees. Here the spines are indexed by the natural numbers, and the changes of measure are based on a more general class of Radon-Nikodym derivatives associated with partitions $\alpha$ of finite subsets $A$ of $\mathbb{N}$. The resulting changes of measure $\Q^{\alpha,T}$ encourages the spines in the set $A$ to flow through the tree in such away that:
\begin{itemize}
\item Spines $\xi^i_T$ and $\xi^j_T$ are following the same particle at time $T ~\iff~ i$ and $j$ are in the same block of $\alpha$.
\item The carriers of spines at time $T$ represent a uniform sample of $|\alpha|$ distinct particles at time $T$.
\end{itemize}
Moreover, in inverting the changes of measure, rather than using Campbell's formula as in \cite{HJR17}, we use a more concise (but arguably less natural) inversion formula -- Lemma \ref{Betalem} -- based around the beta integral.
In summary, though our approach in Section \ref{SpinesSec} to proving Theorem \ref{MT FDD} is perhaps less intuitive, it works in the more general non-binary setting, and in fewer pages. \\

Having proved Theorem \ref{MT FDD}, our methods for deriving the remainder of our fixed-$T$ results, built on generalisations of the Fa\`a di Bruno formula, are altogether different from those used in \cite{HJR17}. Our asymptotic-$T$ results are consequences of Theorem \ref{MT FDD}, and though the critical case has already appeared in \cite{HJR17}, the results in the supercritical and subcritical cases are new.

\subsection{Organisation of the paper}
The rest of the paper is structured as follows. In Section \ref{SpinesSec} we introduce multiple spines and a collection of changes of measure, ultimately leading to a proof of Theorem \ref{MT FDD} under a moment assumption. In Section \ref{Sec5}, we lift this moment assumption, proving the fixed-$T$ results of Section \ref{SecLaw} in full generality, and we also prove the projectivity result Theorem \ref{MT projection}. In Section \ref{Sec6}, we give proofs of the asymptotic results of Section \ref{SecAsy}.

\section{Spines partitions and changes of measure}\label{SpinesSec}

In this section we introduce spines, our tool for calculating the distributions of genealogical trees associated with uniformly chosen particles. For each $n \in \mathbb{N}$, we associate a line of descent $(\xi^n_t)_{t \geq 0}$ that flows through a continuous-time Galton-Watson tree forward in time, choosing uniformly a branch to follow next at branching points. We call this line of descent $(\xi^n_t)_{t \geq 0}$ the $n$-spine. The idea of this section is to create a change of measure under which the first $k$ spines $(\xi^1_t,\ldots,\xi^k_t)$ flow through a tree forward in time in such a way that the subtree they trace out is equal in law to $(\pi^{k,L,T}_t)_{t \in [0,T]}$.\\

In order to avoid confusion about different measures, throughout Section \ref{SpinesSec} we write $\P[\cdot]$ rather than $\E[\cdot]$ for the expectation operator associated with $\P$. Moreover, through this section we will assume that $\P[N_t^k] < \infty$ for all $t$. In Section \ref{no moment} we show this assumption can be lifted, and the results we derive continue to hold.\\

Section \ref{SpinesSec} of the present paper initially runs in parallel to \cite[Section 4]{HJR17}, in some cases generalising the results that hold there for a specific partition to any partition of a finite subset of $\mathbb{N}$. More specifically, Lemma \ref{condform} of the present paper is lifted directly from \cite[Lemma 13]{HJR17}, and the derivation of the following equation, \eqref{QgivenF}, replicates the derivation of \cite[Equation (10)]{HJR17}. We also mention that the special cases $\alpha = \left\{ \{1\},\{2\},\ldots,\{k\} \right\}$ of both Lemma \ref{gproj} and Lemma \ref{Quniform} of the present paper appear as \cite[Lemma 6]{HJR17} and \cite[Lemma 14]{HJR17} respectively.

\subsection{Spines indexed by $\mathbb{N}$}
Suppose under a measure $\P$ we have continuous-time Galton-Watson tree with offspring distribution $L$. Recall we write $\mathcal{N}_t$ for the set of particles alive at time $t$, and $N_t = |\mathcal{N}_t|$. For technical reasons, we append a cemetery state $\Delta$ to the statespace, and write $\hat{\mathcal{N}}_t = \mathcal{N}_t \cup {\Delta}$.  
\\
\\Additionally under $\P$, for each $n \in \mathbb{N}$, there is a right-continuous stochastic process $(\xi^n_t)_{t \geq 0 }$ called the $n$-\emph{spine} defined as follows. 
\begin{itemize}
\item At each time $t \geq 0$, the $n$-spine takes values in $\hat{\mathcal{N}}_t$ -- that is $\xi^n_t \in \hat{\cN}_t$. If $u \in \mathcal{N}_t$ and $\xi^n_t = u$, we say that the $n$-spine is \emph{following} $u$, and that $u$ is \emph{carrying} the $n$-spine. 
\item If a particle carrying the $n$-spine just before time $t$ dies at time $t$ and is replaced by $p \geq 1$ particles $v_1,\ldots,v_p$, then the $n$-spine chooses uniformly among the $p$ offspring a particle to follow next. If the particle carrying the $n$-spine dies at a time $t$ and is replaced by no offspring, we send the $n$-spine to the cemetery state $\Delta$ for the remainder of time. That is, $\xi^n_r = \Delta$ for all $r \geq t$.
\item The $n$-spines don't affect the behaviour of the particles they are following. That is, if a particle $u$ is carrying the $n$-spine at time $t$, then this particle still branches at rate $1$ and has offspring distributed like $L$.
\item The set of $n$-spines $\{ (\xi^n_t)_{t \geq 0} : n \in \mathbb{N} \}$ are independent of one another - that is, if a particle carrying some spines dies and is replaced by $p$ offspring, each of these spines chooses uniformly an offspring to follow next, independently of the others.
\end{itemize}
So in essence, under $\P$ the $n$-spines are simply a set of labels that flow forward in time through a continuous-time Galton-Watson tree without affecting the law of the underlying tree. \\
\\
\hspace*{1.7cm}%
\begin{tikzpicture}[xscale=2,yscale=0.7]
\draw[thick] (0,-5) node(aa){Time 0};
\draw[thick] (5,-5) node(aa){Time t};
\draw[thick] (0,-9.5) -- (0,-5.5);
\draw[thick, dashed] (5,-9.5) -- (5,-5.5);
\draw[thick, gray](0,-7.5) -- (3,-7.5) -- (3, -6.5) -- (3, -8.5) -- (5, - 8.5);
\draw[thick, gray](3,-6.5) -- (3.5,-6.5) -- (3.5,-6) -- (3.5, -7.5) -- (5, -7.5);
\draw[thick, gray] (3.5,-6) -- (4.2, -6) -- (4.2, -5.8) -- (4.2, - 6.4) -- (5,-6.4);
\draw[thick, gray](4.2, -5.8) -- (5,-5.8);
\draw[thick, gray] (3.5, -6.8) -- (3.9, - 6.8);
\draw[thick, red] (3.5,-8.8) node(person2){\small 3458\ldots};
\draw[thick, red] (1.5,-7.2) node(person2){\small 123456789\ldots};
\draw[thick, red] (3.2,-6.3) node(person2){\small 12679..};
\draw[thick, red] (3.8,-5.8) node(person2){\small 17..};
\draw[thick, red] (3.8,-6.6) node(person2){\small 29..};
\draw[thick, red] (3.8,-7.3) node(person2){\small 6..};
\draw[thick, red] (4.4,-6.7) node(person2){\small 1..};
\draw[thick, red] (4.4,-5.6) node(person2){\small 7..};

\draw[thick, red, ->] (2.5,-7.4) -- (3.05, - 7.4) -- (3.05, -6.65) -- (3.3,-6.65) ;
\draw[thick, red, ->] (2.5,-7.6) -- (3.05, - 7.6) -- (3.05, -8.35) -- (3.3,-8.35) ;
\node[align=center,font=\bfseries, yshift=-2em] (title) 
    at (current bounding box.south)
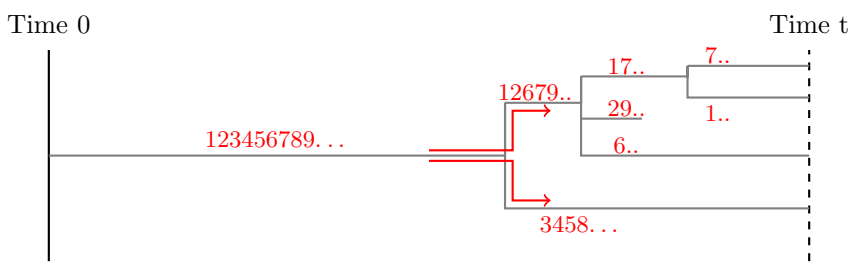
    {Figure 1. The spines flow through the tree forward in time.};
\end{tikzpicture}

\subsection{The spine partition change of measure $\Q^{\alpha,T}$}
For any set $S$ and $k\ge 0$, let $S^{(k)}$ be the set of distinct $k$-tuples from $S$, and for $n\ge0$, write
\[n^{(k)} = \begin{cases} n(n-1)(n-2)\ldots (n-k+1) & \hbox{ if } n\ge k\\ 0 &\hbox{ otherwise.}\end{cases}\]
Note that $|S^{(k)}| = |S|^{(k)}$. Let $\mathcal{F}^\varnothing_t$ be the $\sigma$-algebra containing all the information about the underlying continuous-time Galton-Watson tree up until time $t$, but without any knowledge of which particles the spines are following.
For a subset $A$ of $\mathbb{N}$, we call the set of processes $\{ (\xi_t^A)_{t \geq 0} : a \in A \}$ the $A$-spines. Let $(\mathcal{F}^A_t)_{t \geq 0}$ be the filtration containing all the information about the underlying tree and the carriers of $A$-spines until time $t$:
\begin{align*}
\mathcal{F}^A_t = \sigma\Big( \mathcal{F}^\varnothing_t ; (\xi^a_s)_{ s \in [0,t] }, a \in A \Big).
\end{align*}
We note that our notation is consistent with taking $A$ to be the empty set $\varnothing$. Furthermore, if $B \subset A$, then $\mathcal{F}^B_t \subset \mathcal{F}^A_t$ for each $t$, and in particular, $\mathcal{F}^\varnothing_t \subset \mathcal{F}^A_t$.
\\
\\We now examine the probabilities, conditional on $\mathcal{F}^\varnothing_T$-knowledge, that a given spine is following a given particle in $\mathcal{N}_T$. For a particle in $u \in \mathcal{N}_T$, let $Q(u)$ be the product of offspring sizes of ancestors of $u$:
\begin{align*}
Q( u ) = \prod_{ v < u } L_v. 
\end{align*}
Note that for the $a$-spine to be following particle $u \in \mathcal{N}_T$, for each strict ancestor $v$ of $u$, the $a$-spine must have chosen the `correct' offspring of the $L_v$ offspring of $v$ to continue following. Hence
 \begin{align}\label{SpineProb}
\P ( \xi^{a}_T = u | \mathcal{F}^\varnothing_T ) = Q(u)^{-1}.
\end{align} 
Since the spines behave independently of one another, the probability that the $A$-spines are following a list $( u_a: a \in A)$ of (possibly non-distinct) members of $\mathcal{N}_t$ is
\begin{align} \label{MSpineProb}
\P \big( \cap_{ a \in A} \{ \xi^a_t = u_a \} \big| \mathcal{F}^\varnothing_t \big) = \prod_{a \in A } Q(u_a)^{-1}.
\end{align}
We emphasise that since in general, the quantities $Q(u)^{-1}$ in \eqref{SpineProb}
vary for different $u \in \mathcal{N}_T$, under $\P$ the spines are more likely to be following some particles than others. This is illustrated in Figure 2.
\\ 
\\ \begin{tikzpicture}[xscale=1.8,yscale=0.45]
\draw[thick] (0,-5) node(aa){Time 0};
\draw[thick] (6,-5) node(aa){Time T};
\draw[thick] (0,-9.5) -- (0,-5.5);
\draw[thick] (6,-9.5) -- (6,-5.5);
\draw[thick, gray] (0,-8) -- (2,-8);
\draw[thick, gray] (4,-6.5) -- (2,-6.5) -- (2,-9) -- (6,-9); 
\draw[thick, gray] (4,-6) -- (4, -7.5) -- (6,-7.5) ;
\draw[thick, gray] (4,-6.75) -- (5,-6.75);
\draw[thick, gray] (4,-6) -- (6,-6);
\draw[thick] (6.8,-6) node(person2){\small $u_1~~~ Q(u_1)^{-1} = 1/6$};
\draw[thick] (6.8,-7.5) node(person2){\small $u_2~~~ Q(u_2)^{-1} = 1/6$};
\draw[thick] (6.8,-9) node(person2){\small $u_3~~~ Q(u_3)^{-1} = 1/2$};
\node[align=center,font=\bfseries, yshift=-2em] (title) 
    at (current bounding box.south)
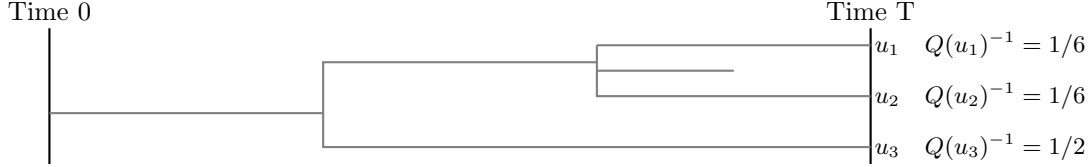
    {Figure 2. The respective probabilities that a given spine is following each \\of $\{u_1,u_2,u_3\}$ at time $T$};
\end{tikzpicture}
\vspace{5mm}

Define the $(\mathcal{F}^A_t)_{t \geq 0}$-adapted, $(\Pi^A \cup \{\Delta\})$-valued process $(\theta^A_t)_{t \geq 0 }$ as follows. If there is an element $a$ of $A$ such that $\xi^a_t = \Delta$, set $\theta^A_t = \Delta$. Otherwise, (that is, if every spine in $A$ is following a living particle at time $t$) let $\theta^A_t$ be the partition of $A$ defined by the equivalence relation
\begin{align*}
a \sim_t b \iff \text{Spines $a$ and $b$ are following the same particle at time $t$, i.e. $\xi^a_t = \xi^b_t \in \mathcal{N}_t$}.
\end{align*}
For the remainder of Section \ref{SpinesSec}, fix a finite subset $A$ of $\mathbb{N}$ and fix a partition $\alpha$ of $A$ into $k$ blocks. Define
\begin{align}\label{Defg}
\hat{\zeta}_{\alpha,t}= \ind\{ \theta^A_t = \alpha \} \prod_{a \in A} Q( \xi^a_t) .
\end{align}
The following lemma gives the $\P$-conditional expectation of $\hat{\zeta}_{\alpha,t}$ given $\mathcal{F}^\varnothing_t$.

\begin{lem} \label{gproj}
For any $t \geq 0$,
$\P[\hat{\zeta}_{\alpha,t}|\F^\varnothing_t] = N^{(k )}_t.$
\end{lem}

\begin{proof}
Note that if $\alpha = \{ A_1,\ldots,A_k \}$, we can decompose the event $\{ \theta_t^A = \alpha \}$ into the disjoint union
\begin{align*}
\{ \theta_t^A = \alpha \} = \bigcup_{ (u_1,\ldots,u_k) \in \mathcal{N}_T^{(k)} } \cap_{i=1}^k \cap_{ a \in A_i} \{ \xi^a_t = u_i \}.
\end{align*}
It follows that 
\begin{align*}
\P[\hat{\zeta}_{\alpha,t}|\F^\varnothing_t]  &= \P\bigg[  \ind\{ \theta^A_t = \alpha \} \prod_{a \in A} Q(\xi^a_t)   \bigg| \mathcal{F}^\varnothing_t \bigg] 
\\&= \P\bigg[ \sum_{u\in\Nc_t^{(k)}} \prod_{i = 1}^k \prod_{ a \in A_i }  \ind \{ \xi^a_t = u_i \}  Q(\xi^a_t) \bigg| \F^\varnothing_t\bigg]\\
&= \sum_{u\in\Nc_t^{(k)}} \P \bigg[\prod_{i = 1}^k \prod_{ a \in A_i }  \ind \{ \xi^a_t = u_i \}  Q(u_i)  \bigg| \mathcal{F}^\varnothing_t \bigg]
\end{align*}
since $\mathcal{N}_t \in \mathcal{F}^\varnothing_t$. Now since the spines are independent,
\begin{align*}
\P \bigg[\prod_{i = 1}^k \prod_{ a \in A_i } \ind \{ \xi^a_t = u_i \}  Q(u_i) \bigg| \mathcal{F}^\varnothing_t \bigg] = \prod_{i = 1}^k \prod_{ a \in A_i } \P \left[ \ind \{ \xi^a_t = u_i \}  Q(u_i)\bigg| \mathcal{F}^\varnothing_t \right].
\end{align*}
Now, $Q(u_i) \in \mathcal{F}_t^\varnothing$, so for every $a,i$,
\begin{align*}
 \P \left[ \ind \{ \xi^a_t = u_i \}  Q(u_i) \bigg| \mathcal{F}^\varnothing_t \right] = Q(u_i)  \P ( \xi^a_t = u_i   | \mathcal{F}^\varnothing_t ) = 1 
\end{align*}
by \eqref{SpineProb}, and hence
\begin{align} \label{RV1}
\P \bigg[\prod_{i = 1}^k \prod_{ a \in A_i } \ind \{ \xi^a_t = u_i \}  Q(u_i) \bigg| \mathcal{F}^\varnothing_t \bigg] = \prod_{i = 1}^k \prod_{ a \in A_i } 1 = 1.
\end{align}
Finally,
\[\P[\hat{\zeta}_{\alpha,t}|\F^\varnothing_t] = \sum_{u\in\Nc_t^{(k)}}   1 = |\Nc_t^{(k)}| = N_t^{(k)}.\]
\end{proof}
Now let $T$ be a fixed time. By the previous lemma $\P[ \hat{\zeta}_{\alpha,T}] = \P [ \P[ \hat{\zeta}_{\alpha,T}| \mathcal{F}^\varnothing_T ]] = \P[ N_T^{(k)}]$, so the random variable 
\[\zeta_{\alpha,T} = \frac{\hat{\zeta}_{\alpha,t}}{\P[N^{(k)}_T]}\]
has unit mean, and we can define a new probability measure $\Q^{\alpha,T}$ on $\mathcal{F}^A_T$ by setting
\begin{equation} \label{francislabel1}
\left.\frac{\d \Q^{\alpha,T}}{\d\P}\right|_{\F^A_T} = \zeta_{\alpha,T}.
\end{equation}
Moreover, by Lemma \ref{gproj} we have
\begin{equation}\label{RadNik0}
\left.\frac{\d \Q^{\alpha,T}}{\d\P}\right|_{\F^\varnothing_T} = \P[ \zeta_{\alpha,T} | \mathcal{F}^\varnothing_T ] = \frac{N_T^{(k)}}{\P[N_T^{(k)}]} =: Z_{k,T}.
\end{equation}


\subsection{Uniformity properties of $\Q^{\alpha,T}$ }  \label{SecUni} 

The goal of this section is to prove that under $\Q^{\alpha,T}$, conditional on $\mathcal{F}^\varnothing_T$, the particles the $A$-spines are following at the time $T$ are equally likely to be any $k$-tuple alive. We then exploit this property to relate the spine process $(\theta^A_t)_{t \in [0,T]}$ to the process $(\pi^{k,L,T}_t)_{t \in [0,T]}$ associated with choosing $k$ particles uniformly from those alive at time $T$.\\

The following result, giving a relationship between projections and changes of measure, is lifted directly from Harris, Johnston and Roberts \cite[Lemma 13]{HJR17}.

\begin{lem}\label{condform}
Suppose that $\Q$ and $\P$ are probability measures on the $\sigma$-algebra $\tilde{\F}$, and that $\F$ is a sub-$\sigma$-algebra of $\tilde{\F}$. If
\[\left.\frac{\d\Q}{\d\P}\right|_{\tilde{\mathcal{F}}} = \zeta \hspace{4mm} \hbox{ and } \hspace{4mm} \left.\frac{\d\Q}{\d\P}\right|_{\mathcal{F}} = Z,\]
then for any non-negative $\tilde{\mathcal{F}}$-measurable X,
\[Z\Q[X|\mathcal{F}] = \P[\zeta X|\mathcal{F}] \hspace{4mm} \P\hbox{-almost surely}.\]
\end{lem}

\begin{proof}
For any $S \in \F$,
\[ \P[ Z \Q[X | \F] S ] = \Q[ \Q[X | \F] S] = \Q[ X S ] = \P[ \zeta X S ]. \]
Since $ Z \Q[X | \F]$ is $\F$-measurable, it therefore satisfies the definition of conditional expectation of $\zeta X$ with respect to $\F$ under $\P$.
\end{proof}

Working in parallel with the derivation of \cite[Equation (10)]{HJR17}, by applying Lemma \ref{condform} using \eqref{francislabel1} and \eqref{RadNik0}, with $\tilde{\mathcal{F}} = \mathcal{F}^A_T$, and $\mathcal{F} = \mathcal{F}^{\varnothing}_T$, we find that for any non-negative $\F^A_T$-measurable random variable $X$, on the event $\{Z_{k,T}>0\}$,
\begin{equation}\label{QgivenF}
\Q^{\alpha,T}[X|\F^\varnothing_T] = \frac{1}{Z_{k,T}}\P[X\zeta_{\alpha,T}|\F^\varnothing_T].
\end{equation}
Note $\zeta_{\alpha,T}$ is supported on $\{ \theta^A_T = \alpha \}$, and hence $
\Q^{\alpha,T}( \theta^A_T = \alpha ) = 1$. In particular, since $\alpha$ partitions $A$ into $k$ blocks, under $\Q^{\alpha,T}$ there must be at least $k$ distinct particles alive at time $T$ for the spines to follow, and hence
\begin{align*}
\Q^{\alpha,T}( N_T \geq k ) = 1.
\end{align*}
In summary, $\Q^{\alpha,T}$-almost surely the $A$-spines at time $T$
are distributed across $k$ different particles in $\mathcal{N}_T$ 
and induce the partition $\alpha$ of $A$ at time $T$.
The following lemma tells us that given knowledge of the tree but not the spines, under $\Q^{\alpha,T}$ the $k$-tuple of carriers of $A$-spines are equally likely to be any $k$-tuple alive.

\begin{lem}\label{Quniform}
The $\Q^{\alpha,T}$-conditional probability given $\mathcal{F}^\varnothing_T$ that the $A$-spines are following a particular $k$-tuple $(u_1,\ldots,u_k) \in \mathcal{N}_T^{(k)}$ 
equals $1/N_T^{(k)}$. That is
\begin{align*}
\Q^{\alpha,T} \Big( \cap_{i = 1}^k \cap_{ a \in A_i }\{ \xi^a_t = u_i \}  \Big| \mathcal{F}^\varnothing_T \Big) = \frac{1}{N_T^{(k)} }. 
\end{align*} 
\end{lem}

\begin{proof}
Note that if $N_T\ge k$ then $Z_{k,T}>0$. Then by \eqref{QgivenF}, for any $u\in \Nc_T^{(k)}$,
\begin{align*}
\Q^{\alpha,T}\bigg(\cap_{i = 1}^k \cap_{ a \in A_i }\{ \xi^a_t = u_i \}  \bigg| \F^\varnothing_T\bigg)  &= \frac{1}{Z_{k,T}} \P \bigg[\zeta_{\alpha,T}\ind\{\cap_{i = 1}^k \cap_{ a \in A_i }\{ \xi^a_t = u_i \}  \} \bigg| \F^\varnothing_T \bigg]
 \\ &= \frac{\P[N_T^{(k)}]}{N_T^{(k)}} \frac{1}{\P[N_T^{(k)}]} \P \bigg[\prod_{i = 1}^k \prod_{ a \in A_i }\ind \{ \xi^a_t = u_i \}  Q(u_i)  \bigg| \mathcal{F}^\varnothing_t \bigg] 
\\ &= \frac{1}{N_T^{(k)}} ,
\end{align*}
where the third equality follows from \eqref{RV1}.
\end{proof}
Given $u=(u_1,\ldots,u_k) \in \mathcal{N}^{(k)}_T$, for each $t \in [0,T]$, let $\pi(u)_t$ be the partition of $\{1,\ldots,k\}$ defined by setting
\begin{align*}
\text{$i$ and $j$ are in the same block of $\pi(u)_t$} \iff \text{$u_i$ and $u_j$ share a common time-$t$ ancestor}.
\end{align*}
Let $(\pi^{k,L,T}_t)_{t \in [0,T]}$ be the partition process associated with picking $k$ distinct particles $U_1,\ldots,U_k$ uniformly from those alive at time $T$ (as defined in the introduction). It follows from the definition of $(\pi^{k,L,T}_t)_{t \in [0,T]}$ that
\begin{align} \label{uniform def}
\ind_{N_T \geq k} \P( \pi^{k,L,T}_{t_1} = \gamma_1, \ldots, \pi^{k,L,T}_{t_n} = \gamma_n | \mathcal{F}^\varnothing_T )  = \frac{ \ind_{N_T \geq k} }{N_T^{(k)} } \sum_{ u \in \mathcal{N}_T^{(k)} } \ind \left \{  \pi(u)_{t_1} = \gamma_1, \ldots, \pi(u)_{t_n} = \gamma_n \right\} .
\end{align}
We use the notation $\Q^{k,T}$ for the change of measure $\Q^{ \{ \{1\}, \ldots, \{ k \} \}, T}$ associated with partitioning $\{1,\ldots,k\}$ into singletons, and similarly we write $\theta^k_t := \theta^{\{1,\ldots,k\}}_t$. The following corollary is the main idea of this section, relating the genealogical process $(\pi^{k,L,T}_t)_{t \in [0,T]}$ to the spine process $(\theta^k_t)_{t \in [0,T]}$.

\begin{cor} \label{shreve} 
We have the identity
\begin{align*}
\ind_{N_T \geq k} \P( \pi^{k,L,T}_{t_1} = \gamma_1, \ldots, \pi^{k,L,T}_{t_n} = \gamma_n | \mathcal{F}^\varnothing_T )   = \ind_{N_T \geq k} \Q^{k,T} \left( \theta^k_{t_1} = \gamma_1, \ldots, \theta^k_{t_n} = \gamma_n | \mathcal{F}^\varnothing_T \right) .
\end{align*}
\end{cor}
\bp
Let $u = (u_1,\ldots,u_k) \in \mathcal{N}_T^{(k)}$. Note that on the event $\{ (\xi^1_T,\ldots,\xi^k_T) = (u_1,\ldots,u_k) \}$ (written $\{ \xi^{(k)}_T = u \}$ below) the processes $(\theta_t^k)_{t \in [0,T]}$ and $(\pi(u)_t)_{t \in [0,T]}$ are identical. Using this fact in the second equality below, we have
\begin{align*}
&\ind_{N_T \geq k}  \Q^{k,T} \left( \theta^k_{t_1} = \gamma_1, \ldots, \theta^k_{t_n} = \gamma_n | \mathcal{F}^\varnothing_T \right)\\  &= \ind_{N_T \geq k}  \Q^{k,T} \left[ \sum_{ u \in \mathcal{N}^{(k)}_T} \ind\{ \xi^{(k)}_T = u \} \ind \left \{ \theta^k_{t_1} = \gamma_1, \ldots, \theta^k_{t_n} = \gamma_n \right\} \Big| \mathcal{F}^\varnothing_T \right]\\
&=\ind_{N_T \geq k}   \Q^{k,T} \left[ \sum_{ u \in \mathcal{N}^{(k)}_T} \ind\{ \xi^{(k)}_T = u \} \ind \left \{\pi(u)_{t_1} = \gamma_1, \ldots, \pi(u)_{t_n} = \gamma_n \right\} \Big| \mathcal{F}^\varnothing_T \right].
\end{align*}
Using the property that $\mathcal{N}_T$ is $\mathcal{F}^\varnothing_T$ is measurable in the first equality below, and Lemma \ref{Quniform} in the second, we yield
\begin{align*}
&=  \ind_{N_T \geq k}  \sum_{ u \in \mathcal{N}^{(k)}_T}  \ind \left \{\pi(u)_{t_1} = \gamma_1, \ldots, \pi(u)_{t_n} = \gamma_n \right\} \Q^{k,T} \left[  \ind\{ \xi^{(k)}_T = u \}  \Big| \mathcal{F}^\varnothing_T \right]\\
&=  \frac{\ind_{N_T \geq k} }{N_T^{(k)} } \sum_{ u \in \mathcal{N}^{(k)}_T}  \ind \left \{\pi(u)_{t_1} = \gamma_1, \ldots, \pi(u)_{t_n} = \gamma_n \right\}.
\end{align*}
By \eqref{uniform def}, this proves the result. 
\ep


\subsection{The joint law of $N_T$ and $(\theta^A_t)_{t \in [0,T]}$ under $\Q^{\alpha,T}$.} \label{SecJoint}
The following theorem gives the joint law of the process $(\theta^A_t)_{t \in [0,T]}$ and $N_T$.

\begin{thm} \label{keytheorem}
Let $\bga = (\gamma_1,\ldots,\gamma_n)$ be a chain of partitions of $A$ such that $\gamma_n \prec \alpha$. Then 
\begin{align} \label{bobo}
\Q^{\alpha,T}( s^{N_T} , \theta^A_{t_1} = \gamma_1, \ldots, \theta^A_{t_n} = \gamma_n) = \frac{ s^k }{ \P[ N_T^{(k)} ] } \prod_{i=0}^n \prod_{ \Gamma \in \gamma_i } F_{\Delta t_i}^{b_i(\Gamma) } (F_{T-t_{i+1}}(s) ) ,
\end{align}
where $(b_i(\Gamma):i = 0,1,\ldots,n)$ are the fragmentation numbers associated with the partition sequence $(\gamma_0,\gamma_1,\ldots,\gamma_n,\gamma_{n+1})$ obtained by setting $\gamma_0 := \{ A \}$ and $\gamma_{n+1} := \alpha$.
\end{thm}
\bp
We proceed by induction. The case $n=0$ follows immediately from \eqref{RadNik0}, since
\begin{align*}
\Q^{\alpha,T}(s^{N_T}) = \P\left[ \frac{N_T^{(k)} }{ \P[N_T^{(k)}] } s^{N_T} \right] = \frac{ s^k }{ \P[N_T^{(k)} ] } F_T^k(s).
\end{align*}
Now we consider the general case $n\geq 0$. Using the definition \eqref{francislabel1} in the first equality below and the tower property in the second,
\begin{align*}
\Q^{\alpha,T}( s^{N_T}, \theta^A_{t_1} = \gamma_1, \ldots, \theta^A_{t_n} = \gamma_n) &= \frac{1}{ \P[N_T^{(k)} ] } \P\left[ \ind \{ \theta_T^A = \alpha \}  \prod_{a \in A} Q(\xi^a_T) ~ \ind\{ \theta^A_{t_1} = \gamma_1, \ldots, \theta^A_{t_n} = \gamma_n \} s^{N_T} \right]\\
&= \frac{1}{ \P[N_T^{(k)} ] } \P\left[ \P \left[ \ind \{ \theta_T^A = \alpha \}  \prod_{a \in A} Q(\xi^a_T) ~ \ind\{ \theta^A_{t_1} = \gamma_1, \ldots, \theta^A_{t_n} = \gamma_n \} s^{N_T} \Big| \mathcal{F}^A_{t_n} \right] \right]\\
&= \frac{1}{ \P[N_T^{(k)} ] } \P\left[ H(\bga,s)_{t_n} \right],
\end{align*}
where
\begin{align*}
H(\bga,s)_{t_n} &= \left( \prod_{a \in A} Q(\xi^a_{t_n} ) \right) \ind\{ \theta^A_{t_1} = \gamma_1, \ldots, \theta^A_{t_n} = \gamma_n \}  \P \left[  \prod_{a \in A} \frac{Q(\xi^a_T)}{ Q(\xi^a_{t_n} ) } \ind \{ \theta_T^A = \alpha \}  s^{N_T} \Big| \mathcal{F}^A_{t_n} \right].
\end{align*}
Note that on the event $\{ \theta_{t_n}^A = \gamma_n \}$, the set $\{ \xi^a_{t_n} : a \in \Gamma \}$ is a singleton for each block $\Gamma$ of $\gamma_n$. In particular, on $\{ \theta_{t_n}^A = \gamma_n \}$, we can decompose $\mathcal{N}_T$ into the disjoint union
\begin{align*}
\mathcal{N}_T = \left( \bigcup_{\Gamma \in \gamma_n } \mathcal{N}^{\Gamma}_{t_n,T} \right) \cup \left\{ u \in \mathcal{N}_T : \nexists a \in A  : \xi^a_{t_n} \leq u \right\},
\end{align*}
where $\mathcal{N}^{\Gamma}_{t_n,T}$ is the set of time-$T$ particles descended from the single element of $\{ \xi^a_{t_n} : a \in \Gamma \}$ of $\mathcal{N}_{t_n}$. Let $N_{t_n,T}^\Gamma$ be the size of $\mathcal{N}^{\Gamma}_{t_n,T}$ and let $\hat{N}_{t_n,T}$ be the size of $\left\{ u \in \mathcal{N}_T : \nexists a \in A  : \xi^a_{t_n} \leq u \right\}$. Then

\begin{align}
H(\bga,s)_{t_n} &=  \left( \prod_{a \in A} Q(\xi^a_{t_n} ) \right)  \ind\{ \theta^A_{t_1} = \gamma_1, \ldots, \theta^A_{t_n} = \gamma_n \}  \P \left[  \prod_{a \in A} \frac{Q(\xi^a_T)}{ Q(\xi^a_{t_n} ) } \ind \{ \theta_T^A = \alpha \}  s^{N_T} \Big| \mathcal{F}^A_{t_n} \right] \nonumber \\
&= \hat{\zeta}_{\gamma_n,t_n} \ind\{ \theta^A_{t_1} = \gamma_1, \ldots, \theta^A_{t_{n-1}} = \gamma_{n-1} \}   \P \left[  s^{\hat{N}_{t_n,T} } \prod_{\Gamma \in \gamma_n} \prod_{a \in A} \frac{Q(\xi^a_T)}{ Q(\xi^a_{t_n} ) } \ind \{ \theta_T^\Gamma = \alpha^\Gamma \}  s^{N_{t_n,T}^\Gamma} \Bigg| \mathcal{F}^A_{t_n} \right]  \label{tripping}
\end{align}
Now on the event $\{ \theta^A_{t_n} = \gamma_n \}$, since the spines in different blocks of $\Gamma$ are following different particles from $t_n$ onwards, the random variables  
\begin{align*}
\left\{  \prod_{a \in A} \frac{Q(\xi^a_T)}{ Q(\xi^a_{t_n} ) } \ind \{ \theta_T^\Gamma = \alpha^\Gamma \}  s^{N_{t_n,T}^\Gamma}  : \Gamma \in \gamma_n \right\}
\end{align*}
are conditionally independent of each other, of $\hat{N}_{t_n,T}$, and of $\mathcal{F}_{t_n}^A$, and are distributed like copies of $\hat{\zeta}_{\alpha^\Gamma,T-t_n} s^{N_{T-t_n}}$. In particular, using Lemma \ref{gproj} in the second equality below we have
\begin{align} \label{ariana}
 \P \left[  \prod_{\Gamma \in \gamma_n} \prod_{a \in A} \frac{Q(\xi^a_T)}{ Q(\xi^a_{t_n} ) } \ind \{ \theta_T^\Gamma = \alpha^\Gamma \}  s^{N_{t_n,T}^\Gamma} \Bigg| \mathcal{F}^A_{t_n} \right]   
 &= \prod_{ \Gamma \in \gamma_n }  \P \left[  \hat{\zeta}_{\alpha^\Gamma,T-t_n} s^{N_{T-t_n}} \right] \nonumber  \\
&=  \prod_{ \Gamma \in \gamma_n }  \P \left[  N_{T-t_n}^{(|\alpha^\Gamma|)}  s^{N_{T-t_n}} \right] \nonumber  \\
&= \prod_{ \Gamma \in \gamma_n } s^{|\alpha^\Gamma|}  F_{T-t_n}^{|\alpha^\Gamma|}(s).
\end{align}
Moreover, since on the event $\{\theta_{t_n}^A = \gamma_n\}$ there are $N_{t_n} - |\gamma_n|$ particles not carrying an $A$-spine at time $t_n$,
\begin{align} \label{grande}
\P[ s^{\hat{N}_{t_n,T}} | \mathcal{F}^A_{t_n} ] = F_{T-t_n}(s)^{N_{t_n} - |\gamma_n|}. 
\end{align}
Finally, writing $\tilde{s} := F_{T-t_n}(s)$, and then using \eqref{francislabel1} in the first equality below and the inductive hypothesis in the second, we have 
\begin{align} \label{minaj}
& \P \left[ \hat{\zeta}_{\gamma_n,t_n} \ind\{ \theta^A_{t_1} = \gamma_1, \ldots, \theta^A_{t_{n-1}} = \gamma_{n-1} \}    \tilde{s}^{N_t - |\gamma_n|}  \right]\nonumber   \\
&=  \P[ N_{t_n}^{(|\gamma_n|)} ]  \Q^{\gamma_n,t_n} \left[  \ind\{ \theta^A_{t_1} = \gamma_1, \ldots, \theta^A_{t_{n-1}} = \gamma_{n-1} \}    \tilde{s}^{N_t - |\gamma_n|}  \right] \nonumber \\
&=  \prod_{i=0}^{n-1} \prod_{ \Gamma \in \gamma_i } F_{\Delta t_i}^{b_i(\Gamma) } (F_{t_n-t_{i+1}}( \tilde{s} ) ).
\end{align}
Combining \eqref{ariana}, \eqref{grande} and \eqref{minaj}, by taking expectations of \eqref{tripping} we obtain
\begin{align*}
\P[H(\bga,s)_{t_n}] = s^{\sum_{\Gamma \in \gamma_n} |\alpha^\Gamma| } \prod_{i = 0}^{n-1} \prod_{ \Gamma \in \gamma_i } F_{\Delta t_i}^{b_i(\Gamma) } (F_{t_n-t_{i+1}}(\tilde{s} ) )   \prod_{ \Gamma \in \gamma_n } F_{T-t_n}^{|\alpha^\Gamma|}(s)  .
\end{align*}
Note that $|\alpha^\Gamma| = b_n(\Gamma)$, and in particular, $\sum_{\Gamma \in \gamma_n} |\alpha^\Gamma| = k$. Finally, by using the semigroup property $F_{t_n - t_{i+1}}(\tilde{s} ) = F_{T-t_{i+1}}(s)$, \eqref{bobo} follows.
\ep

\subsection{Inverting the change of measure}

\begin{lem} \label{Betalem}
Under a probability measure $\P$, let $N$ be a $\{0,1,\ldots\}$-valued random variable and let $X$ be a $[0,\infty)$-valued random variable. Then
\begin{align} \label{EqBetalem}
\int_0^1  \frac{ (1-s)^{ k - 1 }}{ (k-1)!  }   \P[ X ~ N^{(k)} s^{N- k} ]  ds =  \P[ X \ind_{N \geq k} ].
\end{align}
\bp 
Recall from the definition of the beta function that
\begin{align*} 
\frac{ (x-1)!(y-1)! }{ (x +y -1)!} = \int_0^1 s^{x-1} (1-s)^{y-1} ds.
\end{align*}
It follows that for $n \geq k$ we have the identity
\begin{align} \label{EqBetaRep}
\frac{1}{n^{(k)}} = \frac{ (n-k)!}{ n!} = \frac{1}{(k-1)!} \int_0^1 (1-s)^{k-1} s^{n-k} ds .
\end{align}
By interchanging the order of expectation and integration we may write
\begin{align*}
\int_0^1 \frac{ (1-s)^{ k - 1 } }{ (k-1)!}   \P[ X ~N^{(k)} s^{N- k} ] ds &=  \P \left[ X \ind_{N \geq k }  N^{(k)} ~ \frac{1}{ (k-1)!}  \int_0^1 (1-s)^{ k - 1 }   s^{N- k}   ds \right]
\\ &=  \P [ X \ind_{N \geq k } ].
\end{align*}
\ep

\end{lem}

We are now ready to wrap things together to prove Theorem \ref{MT FDD} under the assumption $\P[N_t^{(k)}]< \infty$ for all $t$. Namely for any partition chain $\gamma_i$ and mesh $(t_i)_{ i \leq n } $ we now show that
\begin{align} \label{EqFDD2}
\P( \pi^{k,L,T}_{t_1} = \gamma_1, \ldots, \pi^{k,L,T}_{t_n} = \gamma_n, ~ N_T \geq k ) = \int_0^1  \frac{(1-s)^{k-1}  }{(k-1)!} \prod_{i = 0}^{ n} \prod_{\Gamma \in \gamma_i } F_{\Delta t_i }^{b_i(\Gamma) }\Big( F_{T- t_{i+1}}(s) \Big) ds. 
\end{align}

\bp[Proof of Theorem \ref{MT FDD} under $k^{\text{th}}$-moment assumption]
By Lemma \ref{Betalem} 
\begin{align}
&\P( \pi^{k,L,T}_{t_1} = \gamma_1, \ldots, \pi^{k,L,T}_{t_n} = \gamma_n, ~ N_T \geq k ) \nonumber \\
&= \int_0^1 \frac{ (1-s)^{k-1}}{ (k-1)!} \P \left[ \ind \left\{  \pi^{k,L,T}_{t_1} = \gamma_1, \ldots, \pi^{k,L,T}_{t_n} = \gamma_n \right\} N_T^{(k)} s^{N_T - k } \right] ds \label{chimp1}.
\end{align}
Recall that $\Q^{k,T}$ is the change of measure associated with the partition $\{ \{1\} , \ldots, \{ k \} \}$. Under the assumption $\P[N_T^{(k)}] < \infty$, using \eqref{RadNik0} in the first equality below and Corollary \ref{shreve} in the second, we have
\begin{align}
&\P \left[ \ind \left\{  \pi^{k,L,T}_{t_1} = \gamma_1, \ldots, \pi^{k,L,T}_{t_n} = \gamma_n \right\} N_T^{(k)} s^{N_T - k } \right] \nonumber \\
&= \P[N_T^{(k)} ] \Q^{k,T} \left[  s^{N_T - k } \P \left( \pi^{k,L,T}_{t_1} = \gamma_1, \ldots, \pi^{k,L,T}_{t_n} = \gamma_n \Big| \mathcal{F}^\varnothing_T \right) \right] \nonumber \\
&= \P[N_T^{(k)} ] \Q^{k,T} \left[  s^{N_T - k } \ind \left \{  \theta^k_{t_1} = \gamma_1, \ldots, \theta^k_{t_n} = \gamma_n\right\} \right]  \label{chimp2}.
\end{align}
Finally, by Theorem \ref{keytheorem}
\begin{align}
  \P[N_T^{(k)} ]  \Q^{k,T} \left[  s^{N_T - k } \ind \left \{  \theta^k_{t_1} = \gamma_1, \ldots, \theta^k_{t_n} = \gamma_n\right\} \right] = \prod_{i = 0}^{ n} \prod_{\Gamma \in \gamma_i } F_{\Delta t_i }^{b_i(\Gamma) }\Big( F_{T- t_{i+1}}(s) \Big). \label{chimp3}
\end{align}
Combining \eqref{chimp1}, \eqref{chimp2} and \eqref{chimp3} yields Theorem \ref{MT FDD} under the assumption $\P[N_t^{(k)}] < \infty$ for all $t$. 
\ep

\section{Proofs of fixed-$T$ results} \label{Sec5}

In the sequel, we will return to writing $\E$ for the expectation associated with the probability measure $\P$. In Section \ref{section fdb} we prove a generalisation of the Fa\`a di Bruno formula for semigroups $(F_t)_{t \geq 0}$, which will be used in Section \ref{no moment} to lift the $\E[N_t^{(k)}] < \infty$ assumption from Theorem \ref{MT FDD}.

\subsection{Generalisations of the Fa\`a di Bruno formula} \label{section fdb}
In this section we prove several generalisations of the Fa\`a di Bruno formula
\begin{align} \label{crisp}
 (f \circ g )^k = \sum_{ \gamma \in \Pi^k } \left( f^{|\gamma|} \circ g \right) \prod_{ \Gamma \in \gamma } g^{|\Gamma|}.
\end{align}

Recall that for $\gamma \in \Pi^{k+j}$, $\gamma|^k \in \Pi^k$ is its projection onto $\{1,\ldots,k\}$. Given a chain $\boldsymbol\gamma = (\gamma_1,\ldots,\gamma_n) \in \Pi_n^{k+j}$, let $\boldsymbol\gamma|^k \in \Pi_n^k$ be the chain defined by projecting $\gamma_i$ onto $\{1,\ldots,k\}$ for each $i=1,\ldots,n$. That is, $(\boldsymbol\gamma|^k)_i := \gamma_i|^k$ for each $i$.\\

Finally, we recall that $h^j$ or $h^j(s)$ will refer to the $j^\text{\tiny{th}} $ derivative of a function $j$. (We write $h(s)^j$ for the  $j^\text{\tiny{th}}$ exponent).

\begin{lem} For each chain $\boldsymbol\gamma = (\gamma_1,\ldots,\gamma_n) \in \Pi^k_n$,
\begin{align} \label{faa projection} 
\left(  \prod_{i=0}^n \prod_{\Gamma \in \gamma_i } g_i^{ b_i(\Gamma)} \circ g_{i+1} \circ \ldots \circ g_n \right)^j = \sum_{ \boldsymbol\eta \in \Pi^{k+j}_n : \boldsymbol\eta|^k = \boldsymbol\gamma } \prod_{ i = 0 }^n \prod_{ H \in \eta _i} g_i^{b_i(H)} \circ g_{i+1} \circ \ldots \circ g_n.
\end{align}
\end{lem}

\bp
First we prove the case $j = 1$. By the Leibniz rule,
\begin{align}
&\left(  \prod_{i=0}^n \prod_{\Gamma \in \gamma_i } g_i^{ b_i(\Gamma)} \circ g_{i+1} \circ \ldots \circ g_n \right)^1 \nonumber \\
&= \sum_{0 \leq i \leq n, \Gamma \in \gamma_i} g_i^{b_i(\Gamma)+1} \circ g_{i+1} \circ \ldots \circ g_n \prod_{ l = i+1}^n g_l^1 \circ g_{l+1} \circ \ldots \circ g_n \nonumber \\
&\times \prod_{ 0 \leq p \leq n, \tilde{\Gamma } \in \gamma_p : (p,\tilde{\Gamma}) \neq (i,\Gamma) } g_p^{b_p(\tilde{\Gamma})} \circ g_{p+1} \circ \ldots \circ g_n . \label{bigone}
\end{align}

We now define a bijection 
\begin{align*}
\eta:\{ (i,\Gamma): 0 \leq i \leq n , \Gamma \in \gamma_i \} \to \{ \boldsymbol\zeta \in \Pi^{k+1}_n : \boldsymbol\zeta|^k = \bga \}
\end{align*}
as follows. Let $\eta_j := \eta(i,\Gamma)_j$ be the $j$th partition in the chain $\eta(i,\Gamma)$. Then
\begin{itemize}
\item for each $j \leq i$, $\eta_j$ is formed from $\gamma_j$ by joining $\{k+1\}$ to the block containing $\Gamma$.
\item for each $j \geq i +1$, $\eta_j$ is formed from $\gamma_j$ by adding the singleton set $\{ k + 1 \}$ to $\gamma_j$. 
\end{itemize}
The map $(i,\Gamma) \mapsto \eta(i,\Gamma)$ is clearly injective. Furthermore to see that it is surjective, for any chain $\boldsymbol\zeta \in \Pi^{k+1}_n$ such that $\boldsymbol\zeta|^k = \boldsymbol\gamma$, let $i$ be the largest $0 \leq j \leq n$ such that $k+1$ is contained in a non-singleton block $\Gamma \cup \{k+1\}$ of $\zeta_j$. Then $\boldsymbol\zeta = \eta(i,\Gamma)$.\\

Now consider the fragmentation numbers of the partition $\boldsymbol\eta = \eta(i,\Gamma)$. The fragmentation numbers of $\boldsymbol\eta$ and $\boldsymbol\gamma$ are identical at the levels $p=0,1,\ldots,i-1$. The fragmentation numbers $( b_i(H) : H \in \eta_i )$ at the level $i$ of $\eta = \eta(i,\Gamma)$ are given by replacing the entry $b_i(\Gamma)$ with $b_i(\Gamma) + 1$. Finally, for $p = i+1,\ldots,n$, the fragmentation numbers $(b_p(H) : H \in \eta_p)$ of $\boldsymbol\eta$ at the level $p$ are obtained by adding an entry $1$ to the end of the array $(b_p(\Gamma) : \Gamma \in \gamma_p)$.\\

It follows that
\begin{align} 
& \prod_{ j = 0 }^n \prod_{ H \in \eta(i,\Gamma)_j} g_j^{b_j(H)} \circ g_{j+1} \circ \ldots \circ g_n \nonumber \\
& =  g_i^{b_i(\Gamma)+1} \circ g_{i+1} \circ \ldots \circ g_n \prod_{ l = i+1}^n g_l^1 \circ g_{l+1} \circ \ldots \circ g_n  \prod_{ 0 \leq p \leq n, \tilde{\Gamma } \in \gamma_p : (p,\tilde{\Gamma}) \neq (i,\Gamma) } g_p^{b_p(\tilde{\Gamma})} \circ g_{p+1} \circ \ldots \circ g_n.  \label{identity}
\end{align}
Using \eqref{identity} in the first equality below, and the fact that $\eta$ is a bijection in the second, we have
\begin{align*}
& \sum_{0 \leq i \leq n, \Gamma \in \gamma_i} g_i^{b_i(\Gamma)+1} \circ g_{i+1} \circ \ldots \circ g_n \prod_{ l = i+1}^n g_l^1 \circ g_{l+1} \circ \ldots \circ g_n \nonumber \\
&\times \prod_{ 0 \leq p \leq n, \tilde{\Gamma } \in \gamma_p : (p,\tilde{\Gamma}) \neq (i,\Gamma) } g_p^{b_p(\tilde{\Gamma})} \circ g_{p+1} \circ \ldots \circ g_n\\
&= \sum_{0 \leq i \leq n, \Gamma \in \gamma_i} \prod_{ j = 0 }^n \prod_{ H \in \eta(i,\Gamma)_j } g_j^{b_j(H)} \circ g_{j+1} \circ \ldots \circ g_n\\
&= \sum_{\boldsymbol\zeta \in \Pi^{k+1}_n : \boldsymbol\zeta|^k = \bga } \prod_{ j = 0 }^n \prod_{Z \in \zeta_j } g_j^{b_j(Z)} \circ g_{j+1} \circ \ldots \circ g_n,
\end{align*}
proving the result for $j =1$. \\

Now we prove the general case $j \geq 1$ by induction. Suppose the result holds for all $j' \leq j$. Then by the inductive hypothesis we have
\begin{align} \label{orange}
\left(  \prod_{i=0}^n \prod_{\Gamma \in \gamma_i } g_i^{ b_i(\Gamma)} \circ g_{i+1} \circ \ldots \circ g_n \right)^{j+1} &= \sum_{ \boldsymbol \eta \in \Pi^{k+j}_n : \boldsymbol \eta|^k = \boldsymbol\gamma } \left( \prod_{ i = 0 }^n \prod_{ H \in \eta _i} g_i^{b_i(H)} \circ g_{i+1} \circ \ldots \circ g_n \right)^1.
\end{align}
Using the case $j=1$ for each term in the sum on the right-hand-side of \eqref{orange}, we obtain
\begin{align}  \label{chelsea}
\left(  \prod_{i=0}^n \prod_{\Gamma \in \gamma_i } g_i^{ b_i(\Gamma)} \circ g_{i+1} \circ \ldots \circ g_n \right)^{j+1}  &= \sum_{ \boldsymbol\eta \in \Pi^{k+j}_n : \boldsymbol\eta|^k = \boldsymbol\gamma }~ \sum_{\boldsymbol\theta \in \Pi_n^{k+j+1} : \boldsymbol\theta|^{k+j} = \boldsymbol\eta } \prod_{ i = 0 }^n \prod_{ \Theta \in \theta _i} g_i^{b_i(\Theta)} \circ g_{i+1} \circ \ldots \circ g_n.
\end{align}
Noting we have the disjoint union
\begin{align} \label{fulham}
\cup_{\boldsymbol\eta \in \Pi^{k+j}_n : \boldsymbol\eta|^k = \boldsymbol\gamma } \{\boldsymbol\theta \in \Pi_n^{k+j+1} : \boldsymbol\theta|^{k+j} = \boldsymbol\eta \} = \{\boldsymbol\theta \in \Pi_n^{k+j+1} : \boldsymbol\theta|^{k} = \boldsymbol\gamma  \},
\end{align}
it follows from using \eqref{fulham} in \eqref{chelsea} that
\begin{align*}
\left(  \prod_{i=0}^n \prod_{\Gamma \in \gamma_i } g_i^{ b_i(\Gamma)} \circ g_{i+1} \circ \ldots \circ g_n \right)^{j+1}  = \sum_{ \boldsymbol \theta \in \Pi^{k+j+1}_n : \boldsymbol \theta|^k = \boldsymbol \gamma } \prod_{ i = 0 }^n \prod_{ \Theta \in \theta _i} g_i^{b_i(\Theta)} \circ g_{i+1} \circ \ldots \circ g_n.
\end{align*}
This proves the result holds for $j+1$, and hence by induction, \eqref{faa projection} holds for every $j \geq 1$. 
\ep

\begin{lem}
Let $g_0, \ldots, g_n$ be $k$-times differentiable. Then
\begin{align} \label{fdb gen}
 (g_0  \circ g_1 \circ \ldots \circ g_n )^k = \sum_{ \boldsymbol\gamma \in \Pi^k_n } \prod_{i = 0}^n \prod_{ \Gamma \in \gamma_i } g_i^{ b_i(\Gamma) } \circ g_{i+1} \circ \ldots \circ g_n.
\end{align}
Furthermore, for any semigroup $(F_t)_{t \geq 0}$ of $k$-times differentiable functions, and any mesh $(t_i)_{i \leq n}$ of $[0,T]$, we have
\begin{align} \label{semigroup FdB 2}
 F_T^k = \sum_{ \boldsymbol\gamma \in \Pi^k_n } \prod_{i = 0}^n \prod_{ \Gamma \in \gamma_i } F_{\Delta t_i }^{ b_i(\Gamma)} \circ F_{T - t_{i+1}}.
\end{align}
In particular, for every $s \in [0,1]$ and every chain $\boldsymbol \gamma$ in $\Pi_n^k$ 
\begin{align} \label{semigroup bound}
 F_T^k(s) \geq \prod_{i = 0}^n \prod_{ \Gamma \in \gamma_i } F_{\Delta t_i }^{ b_i(\Gamma)} \left( F_{T - t_{i+1}}(s) \right) \geq 0 .
\end{align}
\end{lem}
\bp
To see that \eqref{fdb gen} is true for $k=1$, we note that $\Pi^1_n$ contains the single chain $\boldsymbol\gamma^* =( \gamma_1,\ldots,\gamma_n)$ where $\gamma_i = \{ \{ 1 \} \}$ for each $i$, the breakage number of the single block $\{1\}$ in each partition is simply $1$, and the result holds by comparison with the chain rule. For general $k$, by using the case $k=1$ to obtain the first equality below and \eqref{faa projection} to obtain the second, we have
\begin{align*}
\left( g_0 \circ \ldots \circ g_n \right)^k &= \left( \prod_{i=0}^n \prod_{\Gamma \in \gamma^*_i} g_i^{b_i(\Gamma)} \circ g_{i+1} \circ \ldots \circ g_n  \right)^{k-1}\\
&= \sum_{  \boldsymbol\eta \in \Pi_n^k :  \boldsymbol\eta|^1 = \boldsymbol\gamma^* } \prod_{i=0}^n \prod_{H \in \eta_i} g_i^{b_i(H)} \circ g_{i+1} \circ \ldots \circ g_n  \\
&= \sum_{ \boldsymbol\eta \in \Pi_n^k } \prod_{i=0}^n \prod_{H \in \eta_i} g_i^{b_i(H)} \circ g_{i+1} \circ \ldots \circ g_n  ,
\end{align*}
where the final equality holds from the fact that \emph{every} chain $ \boldsymbol\eta $ in $ \Pi_n^k $ satisfies  $\boldsymbol\eta|^1 = \boldsymbol\gamma^*$.\\

Now we prove \eqref{semigroup FdB 2}. If $0 \leq t_1 \leq \ldots \leq t_n \leq T$ is a mesh of $[0,T]$, and $(F_t)_{t \geq 0}$ is a semigroup, then we obtain \eqref{semigroup FdB 2} by setting $g_i := F_{\Delta t_i}$, and noting 
\begin{align*}
g_{i+1} \circ \ldots \circ g_n = F_{\Delta t_{i+1}} \circ \ldots \circ F_{\Delta t_n} = F_{ \Delta t_{i+1} + \ldots + \Delta t_n } = F_{T-t_{i+1}}.
\end{align*}
Finally, to prove \eqref{semigroup bound}, note that for every $s \in [0,1]$, $F_t(s) \in [0,1]$, and for every $j \geq 1$, $F^j_t(s) \geq 0$. It follows that for all $j\geq 1$, and $s \in [0,1]$, $F_{t_1}^j( F_{t_2}(s)) \geq 0$ for all $t_1, t_2 \geq 0$. This shows that the summands in \eqref{semigroup FdB 2} are non-negative, and \eqref{semigroup bound} follows.
\ep

\subsection{Proof of Theorem \ref{MT FDD} without moment assumption} \label{no moment}
In this section we prove that we can relax the assumption that $\E[N_t^{(k)}] < \infty$ for each $t$ and Theorem \ref{MT FDD} continues to hold. In order to use a coupling argument in lifting this condition, first we require the following result.\\

Given a Galton-Watson tree, we call $(N_t)_{t \geq 0}$ the \emph{process} associated with the tree, where $N_t$ is the number of particles in the tree at time $t$.
\begin{lem} \label{AncBounded}
Let $\bar{N}_t = | \bigcup_{s \in [0,t]} \mathcal{N}_s |$ be the number of particles who have ever lived up until time $t$. Then, provided the non-explosion hypothesis \eqref{Non-exp} holds, $\P( \bar{N}_t < \infty ) = 1$. 
\bp
Suppose we have a continuous-time Galton-Watson process $N := (N_t)_{t \geq 0}$ with offspring generating function $f(s) = \E[s^L] = p_0 + p_1 s + s^2g(s)$ satisfying the non-explosion hypothesis. Couple $N$ with another process $M := (M_t)_{t \geq 0}$ with generating function $f^*(s) = (p_0 + p_1 + g(s))s^2$ as follows. Every time an particle in the process $N$ has $0$ or $1$ children, the corresponding particle in the process $M$ has $2$ children. Writing $\bar{M}_t$ for the number who have ever lived until $t$ in the $M$-process, clearly $\P( \bar{N}_t \leq \bar{M}_t ) = 1$, and it is straightforward to verify that  $f^*(s)$ also satisfies the non-explosion hypothesis, and hence $M_t$ is almost surely finite.
\\
\\Consider in the process $M$ that every particle is replaced by at least two particles upon death, and hence there were at most $\frac{1}{2}M_t$ parents of particles alive at time $t$. A similar argument says that there can have been at most $\frac{1}{4} M_t$ grandparents, and so forth. It follows that the we can bound above the number who have ever lived: $\bar{M}_t \leq \sum_{i \geq 0} 2^{-i} M_t = 2M_t$. 
\\
\\Since $2M_t \geq \bar{M}_t \geq \bar{N}_t$, the latter quantity is almost surely finite.
\ep
\end{lem}

The following lemma, a variant of the dominated convergence theorem, will be used in the proofs of Theorem \ref{MT FDD}, Theorem \ref{MTSuper} and Theorem \ref{MTSub}. 
\begin{lem} \label{DOM}
Let $g, (g_n)$, and $h, (h_n)$ be measurable functions on a probability space $(\Omega,\mathcal{A},\mu)$, with $|g_n| \leq h_n$ for all $n$, and such that $g_n \to g$, $h_n \to h$, and $\mu h_n \to \mu h$. Then $\mu g_n \to \mu g$. 
\bp See \cite[Theorem 1.21]{Kal97}.
\ep
\end{lem}

We are now ready to prove that Theorem \ref{MT FDD} holds for every offspring distribution with generating function satisfying hypotheses \eqref{Non-triv} and \eqref{Non-exp}.

\bp[Proof of Theorem \ref{MT FDD}]
Our proof idea as follows. To calculate the distribution of $(\pi^{k,L,T}_t)_{t \in [0,T]}$, first calculate $(\pi^{k,L,T}_{n,t})_{t \in [0,T]}$ from a tree where offspring sizes are bounded by $n$. Any such tree clearly satisfies $\E[L^k] < \infty$, and therefore (\cite[Section III.6]{AN72}) $\E[N_t^{(k)}] < \infty$ for every $t$ and the formula \eqref{EqFDD2} applies. Then we send $n \to  \infty$, showing the formula \eqref{EqFDD2} converge suitably. 
\\
\\Let $L$ be a random variable, and let $\mathsf{Tree}_T$ be the continuous-time Galton-Watson tree run until time $T$ and $(N_t)_{t \in [0,T]}$ be the corresponding process for the number of particles alive. We couple the tree $\mathsf{Tree}_T$ with a tree with bounded branching as follows. Let $\mathsf{Tree}_{n,T}$ be the tree with offspring distribution $L \ind_{ L \leq n}$ taken by replacing any birth of size greater than $n$ $\mathsf{Tree}_T$ with a birth of size zero, and let $(N_{n,t})_{t \in [0,T]}$ be the associated process for the number of particles alive.\\

By Lemma \ref{AncBounded}, $\P(\bar{N}_t < \infty) = 1$, and thus $\P( \bar{N}_t \leq n ) \uparrow 1$ as $n \to \infty$. Note that $\{ \bar{N}_t \leq n \}$ ensures $\{ \mathsf{Tree}_{n,T} = \mathsf{Tree}_{T} \}$, since if at most $n$ particles have ever lived, no particle ever had more than $n$ offspring. It follows that $\P( \mathsf{Tree}_{n,T} = \mathsf{Tree}_{T} ) \uparrow 1$ as $n \to \infty$. In particular, $N_{n,t} \to N_t$ almost surely and hence $\P(N_{n,T} \geq k ) \to \P(N_{T} \geq k )$. 
\\
\\If we pick $k$ particles from $\mathsf{Tree}_{n,T}$ and call the partition process $(\pi_{n,t}^{k,L,T})_{t \in [0,T]}$, it follows that $(\pi_{n,t}^{k,L,T})_{ t \in [0,T]}$ converges in distribution to $(\pi^{k,L,T}_t)_{t \in [0,T]}$ as $n \to \infty$, since the partition processes correspond to subtrees of the trees $\mathsf{Tree}_{n,T}$ and $ \mathsf{Tree}_{T}$ respectively. 
\\
\\It remains to check that for a process $(N_{n,t})_{t \geq 0 }$ with offspring distributed like $L\ind_{L \leq n}$ and generating function $F_{n,t}(s)$, that as $n \uparrow \infty$,
\begin{align} \label{TreeConv}
\int_0^1 \frac{(1-s)^{k-1}}{(k-1)! \P(N_{n,T} \geq k ) }  \prod_{i = 0}^{n} \prod_{\Gamma \in \gamma_i } F_{n,\Delta t_i }^{b_i(\Gamma) }\Big( F_{n,T- t_{i+1}}(s) \Big) ds \\ \to  \int_0^1  \frac{(1-s)^{k-1}}{(k-1)! \P(N_T \geq k ) } \prod_{i = 0}^{n} \prod_{\Gamma \in \gamma_i } F_{\Delta t_i }^{b_i(\Gamma)}\Big( F_{T- t_{i+1}}(s) \Big) ds. \label{TreeConv2}
\end{align}
First, let us establish that for all $(j,t,s)$, that as $n \to \infty$ $F^j_{n,t}(s) \to F^j_{t}(s)$. Now $N_{n,t} \uparrow N_t$ almost surely. If $s = 1$, $F^{j}_{n,t}(1) = \E[ N_{n,t}^{(j)}] \uparrow \E[ N_t^{(j)} ] = F^j_t(1)$ by the monotone convergence theorem. If $s < 1$, then the function $n \mapsto n^{(j)} s^{n-j}$ is bounded for $n \in \{0,1,2,\ldots\}$, and hence $F^{j}_{n,t}(s) = \E [ N^{(j)}_{n,t}s^{N_{n,t} - j } ] \to \E[  N^{(j)}_t s^{N_t - j} ] = F^j_t(s)$ by the bounded convergence theorem.
\\
\\To see the convergence of \eqref{TreeConv} to \eqref{TreeConv2}, by \eqref{semigroup bound} we have the domination relation
\begin{align} 
H(s) : = \frac{ (1-s)^{k-1} F_T^k(s)}{  (k-1)!  \P(N_T \geq k )}   \geq  \frac{ (1-s)^{k-1} \prod_{i = 0}^{n} \prod_{\Gamma \in \gamma_i } F_{\Delta t_i }^{b_i(\Gamma) }\Big( F_{T- t_{i+1}}(s) \Big)  }{  (k-1)!  \P(N_T \geq k )}  =: G(s) \geq 0.
\end{align}
Similarly, 
\begin{align} 
H_n(s) := \frac{ (1-s)^{k-1} F_{n,T}^k(s)}{ (k-1)! \P(N_{n,T} \geq k )}   \geq  \frac{ (1-s)^{k-1} \prod_{i = 0}^{n} \prod_{\Gamma \in \gamma_i } F_{n,\Delta t_i }^{b_i(\Gamma)}\Big( F_{n,T- t_{i+1}}(s) \Big)  }{  (k-1)!  \P(N_{n,T} \geq k )}  =: G_n(s) \geq 0.
\end{align}
By our assumption $\P( N_T \geq k ) > 0$, and hence by setting $X=1, N = N_T$ in Lemma \ref{Betalem}, for every $n$, $\int_0^1 H_n(s) ds = 1 =  \int_0^1 H(s) ds$. Trivially, $\int_0^1 H_n(s) ds \to \int_0^1 H(s) ds$. So $G_n(s) \to G(s)$ pointwise, $H_n(s) \to H(s)$ pointwise, $\int_0^1 H_n(s) ds \to \int_0^1 H(s) ds$, and $H_n(s) \geq G_n(s) \geq 0$. It follows by Lemma \ref{DOM} that $\int_0^1 G_n(s) ds \to \int_0^1 G(s) ds$.
\ep

\subsection{Proof of Theorem \ref{mixture markov}} \label{SecMixture}

In this section we will prove Theorem \ref{mixture markov}, which states that $(\pi^{k,L,T}_t)_{t \in [0,T]}$ has a representation in terms of a mixture of Markov processes.

\bp[Proof of Theorem \ref{mixture markov}]
In light of Theorem \ref{MT FDD} we may write
\begin{align*} 
\P( \pi^{k,L,T}_{t_1} = \gamma_1, \ldots, \pi^{k,L,T}_{t_n} = \gamma_n ~ | ~ N_T \geq k ) = \int_0^1 m^{k,L,T}(ds) R_{t_1,\ldots,t_n}(ds),
\end{align*}
where
\begin{align} \label{R fdd 2}
R^{k,L,T}_{t_1,\ldots,t_n}( s, \boldsymbol\gamma ) := \frac{ \prod_{i = 0}^n \prod_{ \Gamma \in \gamma_i } F_{\Delta t_i }^{ b_i(\Gamma)} \left(  F_{T - t_{i+1}}(s) \right) }{ F_T^k(s) } ,
\end{align}
and 
\begin{align} \label{R measure}
m^{k,L,T}(ds) :=  \frac{(1-s)^{k-1} F_T^k(s) }{(k-1)!  \P(N_T \geq k) } ds.
\end{align}
First of all, we remark that setting $X = 1$ and $N = N_T$ in Lemma \ref{Betalem} establishes that $m^{k,L,T}(ds)$ is a probability measure on $[0,1]$. Furthermore, $ R^{k,L,T}_{t_1,\ldots,t_n}( s, \cdot) $ is also a probability measure on $\Pi^k_n$. To see this, note that by \eqref{semigroup bound} that $R^{k,L,T}_{t_1,\ldots,t_n}( s, \boldsymbol\gamma ) \geq  0$ for all $\bga \in \Pi_n^k$ and using \eqref{semigroup FdB 2} it follows that
\begin{align*}
\sum_{ \bga \in \Pi^k_n } R^{k,L,T}_{t_1,\ldots,t_n}( s, \boldsymbol\gamma )  =  1.
\end{align*}
It remains to show that $R^{k,L,T}_{t_1,\ldots,t_n}( s, \cdot)$ are the finite dimensional distributions of a Markov process $(\tpi_t)_{t \in [0,T]}$ satisfying the independent blocks property and with transition density given by \eqref{tpi rates}.\\

Now suppose $\bga := (\gamma_1,\ldots,\gamma_n)$  and $\bga' := ( \gamma_1,\ldots,\gamma_n,\gamma_{n+1},\ldots,\gamma_{n+m})$ are partition chains, and $t_1 < \ldots < t_{n+m}$ are times, then
\begin{align}
\frac{ R^{k,L,T}_{t_1,\ldots,t_{n+m}}( s, \boldsymbol\gamma' ) }{ R^{k,L,T}_{t_1,\ldots,t_n}( s, \boldsymbol\gamma )} &= \frac{ \prod_{ i= n}^m \prod_{ \Gamma \in \gamma_i} F_{\Delta t_i }^{ b_i(\Gamma)}  F_{T - t_{i+1}}(s)  }{ \prod_{ \Gamma \in \gamma_n } F_{T-t_n}^{|\Gamma|} (s) } \nonumber\\
&= \prod_{ \Gamma \in \gamma_n } \frac{ \prod_{ i= n}^m \prod_{ \Gamma' \in \gamma^{\Gamma}_i} F_{\Delta t_i }^{ b_i(\Gamma')}  F_{T - t_{i+1}}(s) }{ F_{T-t_n}^{|\Gamma|} (s)} . \label{chopin product}
\end{align}
Since this expression only depends on $(\gamma_1,\ldots,\gamma_n)$ through the partition $\gamma_n$, $R^{k,L,T}_{t_1,\ldots,t_n}( s, \cdot)$ are the finite dimensional distributions of a Markov process. Additionally, since this expression factorises over $\Gamma \in \gamma_n$, this Markov process satisfies the independent blocks property.\\

Finally, we need to show the transition density of $(\tilde{\pi}_t)_{t \in [0,T]}$ is given by \eqref{tpi rates}. Given $\{ \tpi_{t} = \gamma \}$ and a block $\Gamma \in \gamma$, we want to obtain the conditional probability that $\{ \pi_{t'}^\Gamma = \delta \}$ for a later time $t < t' < T$. To this end note that by setting $m = n+1$ in \eqref{chopin product}, and letting $\delta$ be a partition of $\Gamma \in \gamma_n$, we see that this conditional probability is equal to
\begin{align*}
 \frac{ F_{t_{n+1}-t_n}^{|\delta|} ( F_{T-t_{n+1}}(s)) \prod_{ \Delta \in \delta } F_{T-t_{n+1}}^{|\Delta|}(s) }{ F_{T-t_n}^{|\Gamma|}(s) },
\end{align*} 
establishing \eqref{tpi rates}. 
\ep

\subsection{Proof of the split time representation, Theorem \ref{MT Split}}
\bp[Proof of Theorem \ref{MT Split}]
Let $u_1 < \ldots < u_{n}$ be times in $[0,T]$ and let $h_1, \ldots, h_n$ be small positive reals. First we consider the small-$(h_i)$ asymptotics of the quantity
\begin{align*}
\R^{k,L,T}_s ( \tilde{\pi}_{u_i} = \eta_{i-1}, \tilde{\pi}_{u_i + h_i} = \eta_i, ~~~ \forall i = 1,\ldots,n).
\end{align*}
For $i = 1,\ldots,n$, let $\boldsymbol\gamma = (\gamma_1, \ldots, \gamma_{2n})$ be given by $\gamma_{2i - 1} = \eta_{i-1}$ and $\gamma_{2i} = \eta_i$, and let $t_{2i-1} = u_i$ and $t_{2i} = u_i + h_i$. Then
\begin{align*}
\R^{k,L,T}_s ( \tilde{\pi}_{u_i} = \eta_{i-1}, \tilde{\pi}_{u_i + h_i} = \eta_i, ~~~ \forall i = 1,\ldots,n) = R^{k,L,T}_{t_1,\ldots,t_{2n}}(s, \gamma) .
\end{align*}
Letting $u_0 = 0, h_0 = 0, u_{n+1} = T$, using \eqref{R fdd 2}, we have
\begin{align}
R^{k,L,T}_{t_1,\ldots,t_{2n}}(s, \boldsymbol\gamma)  = \frac{1}{ F_T^k(s)}    \prod_{i=0}^n F'_{u_{i+1} - (u_i + h_i) } (F_{T-u_{i+1}}(s))^{|\eta_i|} \nonumber \\
\times \prod_{ i=1}^n F_{h_i}^{q_i}( F_{T-(u_i + h_i)} (s)) F_{h_i}'(F_{T-(u_i+h_i)}(s))^{|\eta_i|-1}  .
\end{align}
It is straightforward to verify by the definition of $F_t(s)$ that for small $h$,
\begin{align} \label{small h}
F_h^j(s) =  h f^j(s) + o(h), j \geq 2, ~~~F'_h(s) = 1 + o(1).
\end{align}
Using \eqref{small h} we obtain
\begin{align*}
\lim_{ h_i\downarrow 0} \frac{1}{h_1\ldots h_n} R^{k,L,T}_{t_1,\ldots,t_{2n}}(s, \boldsymbol\gamma)   = \frac{ \prod_{i=0}^n F'_{u_{i+1} - u_i } (F_{T-u_{i+1}}(s))^{|\eta_i|} \prod_{ i=1}^n f^{q_i}( F_{T-u_i} (s)) }{ F_T^k(s)}  .
\end{align*}
Finally, note by the semigroup identity that $F_t'(F_{T-t}(s)) = F_T'(s)/F_{T-t}'(s)$, and that $|\eta_i| - |\eta_{i-1}| = q_i - 1$, and hence
\begin{align*}
\prod_{i=0}^n F_{\Delta u_i}' (F_{T-u_{i+1}}(s))^{|\eta_i|} = F_T'(s) \prod_{ i = 1}^n F_{T-u_i}'(s)^{q_i - 1},
\end{align*}
and hence
\begin{align*}
\lim_{ h_i\downarrow 0} \frac{1}{h_1\ldots h_n} R^{k,L,T}_{t_1,\ldots,t_{2n}}(s, \boldsymbol\gamma)   = \frac{  F_T'(s)\prod_{ i=1}^n f^{q_i}( F_{T-u_i} (s)) F_{T-u_i}'(s)^{q_i - 1} }{ F_T^k(s)}  .
\end{align*}
It follows that for any collection of time intervals $[a_1,b_1],\ldots,[a_n,b_n]$ such that $b_i < a_{i+1}$, 
\begin{align} \label{integral formula}
& \mathbb{R}^{k,L,T}_s( \tpi_{a_i} = \eta_{i-1}, \tpi_{b_i} = \eta_{i} ~\forall~ i = 1,\ldots,n)\\ 
&= \int_{a_1}^{b_1} \ldots \int_{a_n}^{b_n} du_1 \ldots du_n  \frac{  F_T'(s)\prod_{ i=1}^n f^{q_i}( F_{T-u_i} (s)) F_{T-u_i}'(s)^{q_i - 1} }{ F_T^k(s)}.
\end{align}
Now by Theorem \ref{mixture markov}, conditional on the event $\{N_T \geq k \}$, $(\pi^{k,L,T}_t)_{t \in [0,T]}$ has law $\int_0^1 m^{k,L,T}(ds) \mathbb{R}_s^{k,L,T}$. Using this fact in conjunction with \eqref{integral formula} (and Fubini's theorem in the final equality) we have
\begin{align*}
& \P( \pi^{k,L,T}_{a_i} = \eta_{i-1}, \pi^{k,L,T}_{b_i} = \eta_{i} ~\forall~ i = 1,\ldots,n ~ | ~ N_T \geq k )\\ 
&= \int_0^1 m^{k,L,T}(ds) \int_{a_1}^{b_1} \ldots \int_{a_n}^{b_n} du_1 \ldots du_n  \frac{  F_T'(s)\prod_{ i=1}^n f^{q_i}( F_{T-u_i} (s)) F_{T-u_i}'(s)^{q_i - 1} }{ F_T^k(s)}\\
&= \int_{a_1}^{b_1} \ldots \int_{a_n}^{b_n} du_1 \ldots du_n \int_0^1 \frac{ (1-s)^{k-1}}{ (k-1)! \P(N_T \geq k ) } F_T'(s)\prod_{ i=1}^n f^{q_i}( F_{T-u_i} (s)) F_{T-u_i}'(s)^{q_i - 1}  ds.
\end{align*}

\ep

\subsection{Proof of the projective extension, Theorem \ref{MT projection}}

Below we will write $(\pi^k_t)_{t \in [0,T]} := (\pi^{k,L,T}_t)_{t \in [0,T]}$.

\bp[Proof of Theorem \ref{MT projection}]

If on the event $\{ N_T \geq k+j \}$, $k+j$ distinct particles $U_1,\ldots,U_{k+j}$ are chosen uniformly from those alive at time $T$, then the first $k$ of them $U_1,\ldots,U_k$ represent a uniformly chosen sample of $k$-distinct particles. It follows by definition that on the event $\{N_T \geq k+j\}$ that $\left(\pi^{k+j}_t\big|^k\right)_{t \in [0,T]}$ and $(\pi^k_t)_{t \in [0,T]}$ are identical in law.\\

It remains to compute the law of the stochastic process $\left(\pi^{k+j}_t\big|^k\right)_{t \in [0,T]}$. Note we can write
\begin{align*}
&\P( \pi^{k+j}_{t_1}|^k = \gamma_1, \ldots,  \pi^{k+j}_{t_n}|^k = \gamma_n, ~ N_T \geq k + j )\\ &= \sum_{ \boldsymbol\eta \in \Pi^{k+j}_n : \boldsymbol\eta|^k = \boldsymbol\gamma }  \P( \pi^{k+j}_{t_1} = \eta_1, \ldots, \pi^{k+j}_{t_n} = \eta_n ,~ N_T \geq k+j) 
\\ &=  \sum_{ \boldsymbol\eta \in \Pi^{k+j}_n : \boldsymbol \eta|^k = \boldsymbol \gamma } \int_0^1 \frac{ (1-s)^{k+j-1}}{(k+j-1)!} \prod_{ i = 0}^n \prod_{ H \in \eta_i} F_{\Delta t_i}^{ b_i(H)} \left( F_{T-t_{i+1}}(s) \right) ds
\\ &=  \int_0^1 \frac{ (1-s)^{k+j-1}}{(k+j-1)!} \sum_{ \boldsymbol\eta \in \Pi^{k+j}_n :\boldsymbol \eta|^k = \boldsymbol\gamma }  \prod_{ i = 0}^n \prod_{ H \in \eta_i} F_{\Delta t_i}^{ b_i(H)} \left( F_{T-t_{i+1}}(s) \right) ds.
\end{align*}
Setting $g_i := F_{\Delta t_i}$ in Lemma \ref{faa projection}, we have
\begin{align*}
\sum_{ \boldsymbol\eta \in \Pi^{k+j}_n : \boldsymbol\eta|^k = \boldsymbol\gamma }  \prod_{ i = 0}^n \prod_{ H \in \eta_i} F_{\Delta t_i}^{ b_i(H)} \left( F_{T-t_{i+1}}(s) \right) ds = \frac{ \partial^j}{ \partial s^{j}} \left( \prod_{ i = 0}^n \prod_{ \Gamma \in \gamma_i} F_{\Delta t_i}^{ b_i(\Gamma)} \left( F_{T-t_{i+1}}(s) \right) \right),
\end{align*}
and the result follows.
\ep

\section{Proofs of asymptotic-$T$ results} \label{Sec6}

\subsection{Supercritical asymptotics}

When $m = \E[L] > 1$ and the Kesten-Stigum condition $\E[L \log_+ L ] < \infty$ holds, the unit-mean and non-negative martingale $W_t := N_t e^{ - (m-1)t}$ converges to almost-surely to a non-degenerate limit $W := W_\infty$, with the properties that $\E[W] = 1$ and 
\begin{align} \label{W pos}
\{ W > 0 \} = \{ N_t > 0 ~ \forall ~ t \} ~~ \text{almost surely}.
\end{align}
Defining $\varphi:[0,\infty) \to [0,1]$ by $\varphi(v) := \E[ e^{ - v W} ]$, we note that by Fubini's theorem,
\begin{align} \label{W fub}
\E[ W^k e^{ - vW } ] = (-1)^k \varphi^{k}(v).
\end{align}

The following lemma gives us the scaling limit of the generating function derivatives under the change of variable $s = e^{- v e^{ - (m-1)T} }$. 

\begin{lem} \label{tech lemma}
For every $v>0$, $t \geq 0$, and non-negative integer $k$,
\begin{align*}
\lim_{T \to \infty} e^{ - k(m-1)T} F_{T-t}^{k}( e^{ - v e^{ - (m-1)T} }) = (-1)^k e^{ - k(m-1)t} \varphi^k(v e^{ - (m-1)t} ) .
\end{align*}
\end{lem}
\bp
Write
\begin{align*}
e^{ - k(m-1)T} F_{T-t}^{k}( e^{ - v e^{ - (m-1)T} }) &= e^{ - k(m-1)T} \E \left[ N_{T-t}^{(k)} e^{ - v e^{ - (m-1)T} N_{T-t}} \right] \\
&= e^{ - k(m-1)t} \E[ h_{v,t}^T( W_{T-t}) ],
\end{align*}
where $h_{v,t}^T(x) = \left( \prod_{ i = 0}^{k-1} (x - i e^{ - (m-1)(T-t)} ) \right) \exp \left( - v e^{ - (m-1)t} (x-ke^{-(m-1)(T-t)}) \right)$. Note that as $T \to \infty$, $h_{v,t}^T(x)$ converges uniformly on $[0,\infty)$ to 
\begin{align*}
h_{v,t}(x) := x^k e^{ - v e^{ - (m-1)t} x }.
\end{align*}
Since $W_{T-t} \to W$ almost surely as $T \to \infty$, it follows that $h_{v,t}^T(W_{T-t})$ converges almost surely to $h_{v,t}(W)$. Since $h_{v,T}^T(x)$ are bounded as $T$ varies, it follows by the bounded convergence theorem that
\begin{align*}
 \E[ h_{v,t}^T( W_{T-t}) ] \to \E[ h_{v,t}(W) ] = (-1)^k \varphi^k( v e^{ - (m-1)t} ) .
\end{align*}
\ep

We are now ready to prove our main result for supercritical trees. 

\bp[Proof of Theorem \ref{MTSuper}]

By Theorem \ref{MT FDD},
\begin{align} 
&\P( \pi^{k,L,T}_{t_1} = \gamma_1, \ldots, \pi^{k,L,T}_{t_n} = \gamma_n~ | ~ N_T \geq k)\\ &= \frac{1}{\P(N_T \geq k ) } \int_0^1  \frac{(1-s)^{k-1}  }{(k-1)! } \prod_{i = 0}^{ n} \prod_{\Gamma \in \gamma_i } F_{\Delta t_i }^{b_i(\Gamma) }\Big( F_{T- t_{i+1}}(s) \Big) ds.
\end{align}
For a fixed mesh $(t_i)_{i \leq n}$, the only interval $[t_i,t_{i+1}]$ that grows with $T$ is the final one $[t_n,T]$, and hence it is convenient to write
\begin{align} 
 = \frac{1}{\P(N_T \geq k ) } \int_0^1  \frac{(1-s)^{k-1}  }{(k-1)! } \prod_{i = 0}^{ n-1 } \prod_{\Gamma \in \gamma_i } F_{\Delta t_i }^{b_i(\Gamma) }\Big( F_{T- t_{i+1}}(s) \Big) \prod_{ \Gamma \in \gamma_n } F_{T-t_n}^{ | \Gamma | }(s)  ds,
\end{align}
where we note $b_n(\Gamma) = | \Gamma|$ by definition. Applying the change of variable $s = e^{ - v e^{ - (m-1)T} }$, we can write 
\begin{align} 
 \P( \pi^{k,L,T}_{t_1} = \gamma_1, \ldots, \pi^{k,L,T}_{t_n} = \gamma_n ~ | ~ N_T \geq k) =\int_0^\infty G^T(v) dv,
\end{align}
where
\begin{align*}
G^T(v) := e^{ - (m-1)T} e^{ - e^{ - (m-1)T} v } \frac{ (1 - e^ { - e^{- (m-1)T} v})^{k-1} }{ (k-1)! \P(N_T \geq k ) }\prod_{i = 0}^{ n-1 } \prod_{\Gamma \in \gamma_i } F_{\Delta t_i }^{b_i(\Gamma) }\Big( F_{T- t_{i+1}}(e ^{ - e^{ - (m-1)T} v } ) \Big) \\
\times \prod_{ \Gamma \in \gamma_n } F_{T-t_n}^{ | \Gamma|} ( e^{ - e^{ - (m-1)T} v } ) .
\end{align*}
Note by \eqref{W pos} that $\P(N_T \geq k) \to 1 - \varphi(\infty)$, and by using Lemma \ref{tech lemma} it can then be seen that as $T \to \infty$, $G^T(v)$ converges pointwise to
\begin{align*}
G(v) := \frac{ e^{ - k( m-1)t_n} v^{k-1} }{ (k-1)! ( 1 - \varphi(\infty)) } \prod_{i = 0}^{ n-1 } \prod_{\Gamma \in \gamma_i } F_{\Delta t_i }^{b_i(\Gamma) }\Big( \varphi( e^{ - (m-1)t_i} v)  \Big) \prod_{ \Gamma \in \gamma_n } (-1)^{ | \Gamma|} \varphi^{| \Gamma|} ( e^{ - (m-1)t_n} v ).
\end{align*}
It remains to establish that $\int_0^\infty G^T(v) dv \to \int_0^\infty G(v) dv$, which we prove using the dominated convergence theorem, Lemma \ref{DOM}. To this end, let
\begin{align*}
H^T(v) := e^{ - (m-1)T}  e^{ - ve^{-(m-1)T} }  \frac{ (1 - e^ { - e^{- (m-1)T} v})^{k-1} }{ (k-1)! \P(N_T \geq k ) } F_T^k(  e^ { - e^{- (m-1)T} v} ) .
\end{align*}
Setting $s = e^{ - e^{ - (m-1)T} v }$ in \eqref{semigroup bound}, we see that $0 \leq G^T(v) \leq H^T(v)$. Furthermore, $H^T(v)$ converges pointwise to 
\begin{align} 
H(v) := \frac{ (-1)^k  v^{k-1}  \varphi^k(v)}{ (k-1)! ( 1 - \varphi(\infty))}.
\end{align}
Finally, we note that for each $T$, by changing variable $s = e^{ - e^{ - (m-1)T} v }$ and then setting $X = 1$ in Lemma \ref{Betalem}, it can be seen that $H^T(v)dv$ is a probability measure. We now show that $H(v)dv$ is also a probability measure. Using \eqref{W fub} in the first equality below and Fubini's theorem in the second, 
\begin{align*}
\int_0^\infty H(v) dv &= \frac{ 1 }{ (k-1)! ( 1 - \varphi(\infty) ) } \int_0^\infty  v^{k-1} \E[ W^k e^{ - vW} ] dv\\
&= \frac{1}{ (k-1)! (1 - \varphi(\infty) )} \E \left[p(W) \right],
\end{align*}
where by the definition of the gamma integral, $p(w) := \ind_{w > 0} \int_0^\infty (vw)^k e^{ - vw} \frac{dv}{v} = (k-1)! \ind_{w > 0}$. Using $\P(W > 0 ) = 1- \varphi(\infty)$, we obtain
\begin{align}
\int_0^\infty H(v) dv &= \frac{1}{ (k-1)! (1 - \varphi(\infty) )} \E \left[ (k-1)! \ind_{W > 0}  \right] = 1. \label{H eq}
\end{align}
So $H^T(v) \geq G^T(v) \geq 0$, $H(v) \geq G(v)$, $H^T(v)$ converge pointwise to $H(v)$ and we trivially have $\int_0^\infty H^T(v) dv = 1 \to \int_0^1 H(v) dv = 1$. It follows by Lemma \ref{DOM} that $\int_0^\infty G^T(v) dv \to \int_0^\infty G(v) dv$. 
\ep

\subsection{Critical asymptotics}
Recall from Section 2 the Kolmogorov-Yaglom exponential limit law \cite[III.7]{AN72}, which states that when $f'(1) = 1$ 
\begin{align} \label{kolm2}
\lim_{T \to \infty} T \P( N_T > 0 ) = \frac{1}{c}, ~~~ \lim_{T \to \infty} \P \left( \frac{N_T}{cT} >  x \Bigg| N_T > 0 \right) = e^{ -x},
\end{align}
where $c := f''(1)/2$.\\

The following lemma gives us the scaled asymptotics of $F_t(s)$ when $f'(1) = 1$ and $c = f''(1)/2 < \infty$. 

\begin{lem} \label{critical scaling}
For any $\theta \in [0,\infty)$, and $a > 0$, $b \geq 0$, $j \geq 1$,
\begin{align*}
\lim_{T \to \infty} (cT)^{-(j-1)} F_{aT}^j\left( F_{bT}( \exp( - \theta/cT) ) \right) = a^{j-1} j! \left( \frac{ 1 + \theta b }{ 1 + \theta (a+b) } \right)^{j+1} .
\end{align*}
\end{lem}
\bp Throughout this proof, $Z$ will refer to a standard exponential random variable, and we note
\begin{align*}
\E[ Z^j e^{ - \phi Z} ] = \frac{ j!}{ (1 + \phi)^{j+1} }.
\end{align*}
By \eqref{kolm2}, conditioned on $\{N_{bT} > 0\}$, the random variable $N_{bT}/cT$ converges in distribution to $bZ$ as $T \to \infty$. Now since $x \mapsto e^{ - \theta x}$ is bounded for $x \geq 0$, it follows that for any $b > 0$ we have
\begin{align*}
\lim_{T \to \infty} T \left( 1 - F_{bT}( \exp( - \theta/cT) ) \right) &= \lim_{T \to \infty} T( 1 -  \E[ \exp( - \theta N_{bT} /cT) ])\\
&=  \lim_{T \to \infty}  T \P(N_{bT} > 0 ) \Big(  1 - \E[ \exp( - \theta N_{bT} /cT) | N_{bT} > 0 ] \Big)\\
&= \frac{1}{ cb} ( 1 - \E[ e^{ - \theta b Z} ] ),
\end{align*}
where $Z$ is a standard exponential, and hence
\begin{align} \label{eq 11}
\lim_{T \to \infty} T ( 1 - F_{bT}( \exp( - \theta/cT) ))  = \frac{1}{c} \frac{ \theta}{ 1 + \theta b }.
\end{align}
Note that \eqref{eq 11} is also true for $b = 0$, since $F_0(s) = s$. It follows that for any $b \geq 0$,
\begin{align*}
\lim_{T \to \infty} F_{bT}(  \exp( - \theta/cT) )^{cT} = \exp \left( - \frac{ \theta}{ 1+ \theta b} \right) .
\end{align*}
Moreover, conditional on $\{ N_{aT} > 0 \}$,
\begin{align*}
\frac{ N_{aT}^{(j)} }{ (c a T)^{j}  } F_{bT}(  \exp( - \theta/cT) )^{N_{aT} - j} = \frac{ N_{aT}^{(j)} }{ (c a T)^{j}  }F_{bT}(  \exp( - \theta/cT) )^{ca T \times (N_{aT}/caT - j/caT)}
\end{align*}
converges in distribution to $Z^j \exp \left( - \frac{ \theta}{ 1+ \theta b} aZ \right)$. It follows that
\begin{align*}
& \lim_{T \to \infty} (cT)^{-(j-1)} F_{aT}^j\left( F_{bT}( \exp( - \theta/cT) ) \right)\\
&=  a^{j-1} \lim_{T \to \infty} (caT) \P(N_{aT} > 0 )  \lim_{T \to \infty}  \E \left[ \frac{ N_{aT}^{(j)}}{ (caT)^j }  F_{bT}(  \exp( - \theta/cT) )^{N_{aT} - j} \Big| N_{aT} > 0 \right]\\
&=  a^{j-1} \E \left[ Z^j \exp \left( - \frac{ \theta}{ 1+ \theta b} aZ \right)\right]\\
&= a^{j-1} j! \left( \frac{ 1 + \theta b }{ 1 + \theta (a+b) } \right)^{j+1}.
\end{align*}
as required.\\
\ep

We now prove our main result for critical trees.

\bp[Proof of Theorem \ref{MTCrit}]
Let $f'(1) = 1, c = f''(1)/2 < \infty$. Then by \eqref{EqFDD}, for $t_1 < \ldots < t_n \in [0,1]$,
\begin{align*}
\P( \pi^{k,L,T}_{Tt_1} = \gamma_1, \ldots, \pi^{k,L,T}_{Tt_n} = \gamma_n | N_T \geq k) = \frac{1}{ \P(N_T \geq k )} \int_0^1  \frac{(1-s)^{k-1}  }{(k-1)! } \prod_{i = 0}^{ n} \prod_{\Gamma \in \gamma_i } F_{ T \Delta t_i }^{b_i(\Gamma) } \left( F_{T(1 - t_{i+1})}(s) \right) ds,
\end{align*}
Taking the change of variable $s = e^{ - \theta/cT}$, we may write
\[ \P( \pi^{k,L,T}_{Tt_1} = \gamma_1, \ldots, \pi^{k,L,T}_{Tt_n} = \gamma_n | N_T \geq k) = \int_0^\infty G^T(\theta) d\theta,\]
where
\begin{align*}
G^T(\theta) := \frac{ e^{ - \theta/cT} }{ cT \P(N_T \geq k) } \frac{ (cT)^{k-1} (1 - e^{ - \theta/cT})^{k-1}  }{ (k-1)!}  \frac{1}{(cT)^{k-1} } \prod_{i=0}^n \prod_{ \Gamma \in \gamma_i} F_{ T ( \Delta t_i) }^{ b_i( \Gamma)} \left( F_{T(1 - t_{i+1})} ( e^{ - \theta/cT} ) \right) .
\end{align*}
Now note that given any chain of partitions $(\gamma_1 , \ldots , \gamma_{n})$, we have by the definition of the fragmentation numbers $b_i(\Gamma)$,
\begin{align*}
\sum_{ \Gamma \in \gamma_i } (b_i(\Gamma) - 1)= |\gamma_{i+1}| - |\gamma_i|,
\end{align*}
and in particular, $\sum_{ i = 0 }^n \sum_{ \Gamma \in \gamma_i } b_i( \Gamma) = k - 1$. 
Using these facts in conjunction with Lemma \ref{critical scaling} we have for every $\theta$,
\begin{align*}
G^T(\theta) \to G(\theta) := \prod_{i=0}^n \prod_{ \Gamma \in \gamma_i}  b_i(\Gamma)!  \frac{ \theta^{k-1}}{(k-1)!} \prod_{i=0}^n (\Delta t_i)^{|\gamma_{i+1}| - | \gamma_i| } \left( \frac{ 1 + (1-t_{i+1})\theta}{ 1 + (1 - t_i)\theta } \right)^{|\gamma_{i+1}|}.
\end{align*}
It remains to show that $\int_0^\infty G^T(\theta) d \theta \to \int_0^\infty G(\theta) d \theta$. To this end, define 
\begin{align*}
H^T(\theta) := \frac{ e^{ - \theta/cT} }{ cT \P(N_T \geq k) } \frac{ (cT)^{k-1} (1 - e^{ - \theta/cT})^{k-1}  }{ (k-1)!}  \frac{1}{(cT)^{k-1} } F_T^k( e^{ - \theta/cT} ) .
\end{align*}
Setting $s = e^{ - \theta/cT}$ in \eqref{semigroup bound} we see that $H^T(\theta) \geq G^T(\theta) \geq 0$. We also note that by Lemma \ref{critical scaling} that $H^T(\theta)$ converges pointwise to 
\begin{align*}
H(\theta) := \frac{ k \theta^{k-1} }{( 1+ \theta)^{k+1}} .
\end{align*}
Furthermore, by taking the change of variable $s = e^{ - \theta/cT}$ and using Lemma \ref{Betalem}, it can be seen each $H^T(\theta) d\theta$ is a probability measure on $[0,\infty)$. It is also straightforward to verify that $H(\theta) d\theta$ is also a probability measure on $[0,\infty)$.\\

In particular, we trivially have $\int_0^\infty H^T(\theta) d\theta \to \int_0^\infty H(\theta) d\theta$, and it follows from Lemma \ref{DOM} that $\int_0^\infty G^T(\theta) d\theta \to \int_0^\infty G(\theta) d \theta$. 

\ep

\subsection{Subcritical asymptotics}
Finally, we prove our main result in the subcritical case. This proof is more straightforward than the supercritical and critical cases since on survival until a large time $T$, there are only a constant order of particles alive at time $T$, and as a result, no scaling is needed in the generating functions.\\

Recall that when $m <1$, there exist non-negative numbers $\{c_j: j \geq 1 \}$ satisfying $\sum_{j \geq 1} c_j = 1$ such that
\begin{align} \label{quasi2}
\lim_{T \to \infty} \P( N_T = j | N_T > 0 ) = c_j. 
\end{align}
We set $C(s) := \sum_{j \geq 1} c_j s^j$, and note that according to \cite[I.11]{AN72} that $C'(1) < \infty \iff \E[L \log_+ L] < \infty$.

\begin{lem}
Suppose $f'(1) < 1$ and $\E[ L \log_+ L ] < \infty$. Then for all $l \geq 1$,

\begin{align} \label{sub prob}
\lim_{T \to \infty} \frac{ F_{T-t}^l (F_t(s)) }{ \P(N_T \geq k ) }  = \frac{ e^{ - (m-1)t}}{ 1 - \sum_{j=1}^{k-1} c_j} C^l(F_t(s)).
\end{align}
\end{lem}
\bp
First note that
 \begin{align*}
\lim_{T \to \infty} e^{ - (m-1)T} \P(N_T > 0 ) = \lim_{T \to \infty} e^{ - (m-1)T} \frac{ \E[N_T ]}{ \E[N_T | N_T > 0 ] } = \frac{1}{C'(1)}.
\end{align*}
Now write
\begin{align*}
\lim_{T \to \infty} \frac{ F_{T-t}^l (F_t(s)) }{ \P(N_T \geq k ) } =   \lim_{T \to \infty} \frac{ \P(N_T > 0 ) }{ \P(N_T \geq k ) }   \lim_{T \to \infty}  \frac{ \P(N_{T-t} > 0 ) }{ \P(N_T > 0 ) }  \lim_{T \to \infty} \frac{ F_{T-t}^l (F_t(s)) }{ \P(N_{T-t} > 0  ) } ,
\end{align*}
and use the definition \eqref{quasi2}.

\ep

\vspace{4mm}
We now prove our main result for subcritical trees.

\bp[Proof of Theorem \ref{MTSub}]
By replacing $t_i$ with $T - t_{n - i +1}$ in Theorem \ref{MT FDD}, we have the following formula for the finite dimensional distributions of $(\rho^{k,L,T}_t)_{t \in [0,T]}$:
\begin{align} 
\P( \rho^{k,L,T}_{t_1} = \gamma_1, \ldots, \rho^{k,L,T}_{t_n} = \gamma_n | N_T \geq k) 
&= \frac{1}{ \P(N_T \geq k ) } \int_0^1 \frac{(1-s)^{k-1}}{ (k-1)! } \prod_{ i = 1 }^{n+1} \prod_{ \Gamma \in \gamma_i } F_{ \Delta t_{-1}}^{ m_i(\Gamma) } ( F_{t_{i-1}}(s)) ds,\label{rho}
\end{align}
where $\gamma_1 \succ \ldots \succ \gamma_{n}$ is a chain with merger numbers $(m_i(\Gamma))$. \\

Note that $m_{n+1}( \{ 1, \ldots, k \} ) = |\gamma_n|$. Now for fixed $(t_i)_{i \leq n }$, as we send $T \to \infty$ the only time interval $[t_i,t_{i+1}]$ in \eqref{rho} that grows with $T$ is $[t_n,T]$. For this reason it is useful to write
\begin{align*}
&= \int_0^1 \frac{(1-s)^{k-1}}{ (k-1)! } \prod_{ i = 1 }^{n} \prod_{ \Gamma \in \gamma_i } F_{ \Delta t_{i-1}}^{ m_i(\Gamma) } ( F_{t_{i-1}}(s)) ~ \frac{ F_{T-t_n}^{|\gamma_n|} \left( F_{t_n} (s) \right) }{ \P(N_T \geq k )}  ds  =: \int_0^1 G^T(s) ds.
\end{align*}
By \eqref{sub prob}, for each $s \in [0,1]$,
\begin{align*}
G^T(s) \to G(s) := \frac{ e^{ - (m-1)t_n}}{ 1 - \sum_{j=1}^{k-1} c_j} \frac{(1-s)^{k-1}}{ (k-1)! } \prod_{ i = 1 }^{n} \prod_{ \Gamma \in \gamma_i } F_{ \Delta t_{i-1}}^{ m_i(\Gamma) } ( F_{t_{i-1}}(s))  C^{|\gamma_n|}(F_{t_n} (s)).
\end{align*}
It remains to establish that $\int_0^1 G^T(s) ds \to \int_0^1 G(s)ds$. To this end, define
\begin{align*}
H^T(s) := \frac{ (1-s)^{k-1} F_T^k(s) }{ (k-1)! \P(N_T \geq k )} 
\end{align*}
By replacing $t_i$ with $T - t_{n +1 - i}$ and $\gamma_i$ with $\gamma_{n+1 - i}$ (and noting $m_i(\Gamma) = b_{n-i}( \Gamma)$)  in \eqref{semigroup bound}, we obtain
\begin{align*}
H^T(s) \geq G^T(s) \geq 0.
\end{align*}
Furthermore, by applying \eqref{sub prob} with $t=0$ and $l=k$, we see that $H^T(s)$ converges pointwise to 
\begin{align*}
H(s) := \frac{ (1-s)^{k-1}}{ (k-1)! (1 - \sum_{j = 1}^{k-1} c_j) } C^k(s) .
\end{align*}
Finally, by Theorem \ref{mixture markov}, $H^T(s)ds = m^{k,L,T}(ds)$ is a probability measure on $[0,1]$, and furthermore, by Lemma \ref{Betalem}, so is $H(s)ds$. We trivially have $\int_0^1 H^T(s) ds  \to \int_0^1 H(s) ds$, and hence by Lemma \ref{DOM}, $\int_0^1 G^T(s) ds \to \int_0^1 G(s)ds$.

\ep

\section*{Acknowledgements}
The author would like to thank his PhD supervisors Matt Roberts and Simon Harris for their continual insight and support. The author would also like to thank an anonymous referee whose suggestions have improved the paper. This research was supported in the early stages by University of Bath URS funding, and in the latter stages by the ERC grant \emph{Integrable Random Structures} at University College Dublin.


\begin{thebibliography}{79}

\bibliographystyle{imsart-number}
\bibitem{Ath12a}
\textsc{Athreya, K.B.} (2012). Coalescence in the recent past in rapidly growing populations.
\emph{Stoch. Proc. Appl.} \textbf{122} 3757-3766.

\bibitem{Ath12b}
\textsc{Athreya, K.B.} (2012). Coalescence in critical and subcritical Galton-Watson branching processes.
\emph{J. Appl. Prob.} \textbf{49} 627-638.

\bibitem{Ath16}
\textsc{Athreya, K.B.} (2016). Coalescence in Branching Processes. \emph{Branching Processes and Their Applications.} Lecture Notes in Statistics, \textbf{219} Springer.


\bibitem{AN72}
\textsc{Athreya, K.B. and Ney, P.E.} (1972). \emph{Branching Processes}. Springer-Verlag, New York.


\bibitem{buhler:super}
\textsc{B\"uhler, W.} (1971).
Generations and degree of relationship in supercritical Markov branching processes.
\emph{Z. Wahrscheinlichkeitstheorie verw. Geb.} \textbf{18} 141-152.







%

\bibitem{Dur78}
\textsc{Durrett, R.} (1978).
The genealogy  of critical branching processes.
 \emph{Stochastic  Processes  and their Applications} \textbf{8}(1) 101-116.

\bibitem{grosjean_huillet:coalescence}
\textsc{Grosjean, N. and Huillet, T.} (2018).
On the genealogy and coalescence times of Bienayme-Galton-Watson branching processes.
\emph{Stochastic Models,} \textbf{34}(1).






\bibitem{Har63}
\textsc{Harris, T.E.} (1963).
\emph{The Theory of Branching Processes}. Berlin, Springer. 

\bibitem{HJR17}
\textsc{Harris-Johnston-Roberts} (2019).
The coalescent structure of continuous-time Galton-Watson trees.
Preprint: \href{https://arxiv.org/abs/1703.00299}{https://arxiv.org/abs/1703.00299}.

\bibitem{Joh02}
Johnson, W.P.: The Curious History of Fa\`a  Di Bruno's formula.  (2002).
\emph{Am. Math. Mon.} \textbf{109}(3) 217--234.

\bibitem{Kal97}
\textsc{Kallenberg, O.} (1997).
\emph{Foundations of Modern Probability}.
Springer-Verlag, New York.

\bibitem{kesten stigum}
\textsc{Kesten, H. and Stigum, B.P.} (1966)..
 A Limit Theorem for Multidimensional Galton-Watson Processes
{\em Ann. Math. Stat} \textbf{37}(5) 1211--1223.
\
\bibitem{Kin82}
\textsc{Kingman, J.F.C.} (1982).
The Coalescent.
\emph{Stoch. Proc. Appl.} \textbf{13}(3) 235--248.

\bibitem{Lam03}
\textsc{Lambert, A.} (2003).
Coalescence times for the branching process.
\emph{Adv. Appl. Prob.} \textbf{35}(04) 1071--1089. 




\bibitem{Le14}
\textsc{Le, V.} (2014).
Coalescence times for the Bienayme-Galton-Watson process.
\emph{J. Appl. Probab.} \textbf{51} 209--218.


\bibitem{LV}
\textsc{Liu, M. and Vatutin, V.A.} (2019).
Reduced critical branching processes for small populations.
\emph{Theory Probab. Appl.} \textbf{63}(4) 648--656.


\bibitem{Murtagh}
\textsc{Murtagh, F.} (1984).
Counting dendrograms: A survey.
\emph{Discrete Appl. Math.} \textbf{7}(2) 191--199.


\bibitem{OCo95}
\textsc{O'Connell, N.} (1995)
The genealogy of branching processes and the age of our most recent common ancestor.
\emph{Adv. Appl. Probab.} \textbf{27}(02) 418-442.


\bibitem{VHJ} 
\textsc{Vatutin, V.A., Hong, W., and Ji, Y.} (2018).
Reduced critical Bellman-Harris branching processes for small populations.
\emph{Discrete Math. Appl.} \textbf{28}(5) 319--330.


\bibitem{Zubkov}
\textsc{Zubkov, A. M.} (1976).
 Limiting distributions of the distance to the closest common ancestor. 
\emph{Theory Prob. Appl.} \textbf{20}(3), 602--612.






%


\end{thebibliography}
\end{document}